\documentclass[preprint,10pt]{elsarticle}

\usepackage[T1]{fontenc}
\usepackage[utf8]{inputenc}
\usepackage[english]{babel}
\usepackage{graphicx}
\usepackage{amsmath,amssymb,amsfonts}
\usepackage{amsthm}
\usepackage{mathtools}
\usepackage{bm}
\usepackage{mathrsfs}
\usepackage{textcomp}
\usepackage{manyfoot}
\usepackage{booktabs}
\usepackage{enumitem}
\usepackage{microtype}

\biboptions{sort&compress}
\theoremstyle{plain}
\newtheorem{theorem}{Theorem}
\newtheorem{corollary}{Corollary}
\newtheorem{lemma}{Lemma}
\newtheorem{proposition}{Proposition}

\theoremstyle{definition}
\newtheorem{example}{Example}

\theoremstyle{remark}
\newtheorem{remark}{Remark}

\AtBeginDocument{}
\journal{arXiv preprint}

\begin{document}

\begin{frontmatter}

\title{Cone Minimax Principles for  Non-Selfadjoint Operator Pencils}

\author[inst1]{Yavdat Sh. Il'yasov\corref{cor1}}
\ead{ilyasov02@gmail.com}

\author[inst1]{Nur F. Valeev}
\ead{valeevnf@yandex.ru}

\cortext[cor1]{Corresponding author.}

\affiliation[inst1]{organization={Institute of Mathematics with Computing Centre, Ufa Federal Research Centre, Russian Academy of Sciences},
             city={Ufa},
             country={Russia}}

\begin{abstract}
We propose a variational approach to principal spectral values of
	non-selfadjoint operator pencils \(\mathcal L u=\lambda\mathcal G u\), where
	the weight operator \(\mathcal G\) may be singular. The aim is to obtain
	Rayleigh-type minimax formulas for selected real spectral levels in settings
	where the standard selfadjoint variational theory is unavailable and
	positivity-based methods may not apply directly. The construction is based on
	the extended two-variable Rayleigh quotient
	\[	\mathcal R(u,v)
	=
	\frac{\langle \mathcal L u,v\rangle}
	{(\mathcal G u,v)_H},\]
		 defined
	on admissible cone pairs. It leads to dual sup--inf and inf--sup principal
	levels, cone quasi-eigenvalues, and corresponding trapping and saddle-point
	principles. The resulting minimax formulas characterize selected real cone
	levels of non-selfadjoint operator pencils and identify them with principal
	spectral values whenever positive right-left eigenpairs exist, including
	cases with non-invertible operators and singular weights.
	
	We prove that these formulas are stable under finite-dimensional
	approximation. Thus the classical idea of approximating spectral data by
	finite-dimensional variational problems acquires an analogue for
	non-selfadjoint operator pencils in an ordered cone setting. The method also
	yields a posteriori spectral certificates, one-sided perturbation bounds, and
	approximation estimates. Elliptic examples illustrate both the scope of the
	method and the sharpness of the estimates.
\end{abstract}


\end{frontmatter}


\section{Introduction}

We study spectral problems of the form
\[
\mathcal L u=\lambda \mathcal G u,
\]
where \(\mathcal L\) is not necessarily selfadjoint or invertible, and the
weight operator \(\mathcal G\) may be singular. The purpose of this paper is
to develop a cone variational characterization of selected principal levels
for operator pencils \(\mathcal L-\lambda\mathcal G\), based on an extended
two-variable Rayleigh quotient. When positive right-left eigenpairs exist,
these levels coincide with the corresponding principal spectral values.

The spectral theory of non-selfadjoint operators does not usually admit a
variational structure comparable to that of the selfadjoint theory. In the
selfadjoint case, the Rayleigh quotient and the Courant--Fischer principle
provide eigenvalue characterizations and serve as a basis for spectral
approximation and qualitative analysis. For non-selfadjoint operators,
principles of this type are generally unavailable. Principal spectral values
can often be treated by positivity methods, such as Perron--Frobenius and
Krein--Rutman theory~\cite{KreinRutman}, or, for elliptic operators, by
maximum-principle methods and generalized principal eigenvalue theory
\cite{ProtterWeinberger,BerestNienbVardan,Nussbaum1988,LemmensNussbaum2012}.
These methods are powerful, but they typically rely on positivity, cone
invariance, maximum-principle structure, or invertibility of a shifted operator
together with positivity of the corresponding resolvent.

Variational characterizations of principal spectral quantities in ordered
settings have a long history. In finite dimensions, the classical
Collatz--Wielandt and Birkhoff--Varga formulas characterize the Perron root
of positive or essentially positive matrices
\cite{Birkhoff1957,Collatz1942,Varga1962,Wielandt1950}. For nonsymmetric
elliptic operators, Donsker--Varadhan obtained a variational formula for the
principal eigenvalue under maximum-principle assumptions~\cite{Donsker}, and
the connection between principal eigenvalues and maximum principles was
further developed, in particular, by Berestycki--Nirenberg--Varadhan
\cite{BerestNienbVardan}. The present work is close in spirit to this ordered
variational viewpoint, but the object considered here is an operator pencil
\(\mathcal L-\lambda\mathcal G\), the weight \(\mathcal G\) is allowed to be
singular, and the variational quotient has two independent cone variables.

The need for such a formulation already appears in simple systems. A
zero-order skew-symmetric coupling may destroy invariance of the standard
componentwise positive cone, even though the operator still has meaningful
real levels detected by ordered tests. This motivates an approach which treats
the quotient itself as the basic object and does not first reduce the problem
to the construction of a positive resolvent.

A further motivation comes from bifurcation theory. Extended
Rayleigh-quotient type formulas, close in spirit to the one used here, have
been applied to the localization of maximal saddle-node bifurcations in
nonlinear equations; see
\cite{ilyasov2021,IlyasJDE24,Ilyasov2026,Salazar2020}. In such problems one
seeks a variational or minimax characterization of the critical parameter at
which solvability or stability changes. The present paper develops a linear
spectral counterpart of this idea for non-selfadjoint operator pencils.

Let \(\mathcal S\) be a cone and let \(\mathcal S^o\) denote its strictly
positive part. Following the extended functional method introduced in
\cite{IlyasFunc} and further developed in
\cite{IlyasJDE24,IlyasovValeev2024}, we use the two-variable Rayleigh quotient
\[
\mathcal R(u,v)
=
\frac{\langle\mathcal L u,v\rangle}{(\mathcal G u,v)_H},
\]
where \(u\) and \(v\) are independent cone variables. The denominator is
assumed positive on admissible mixed cone pairs. We consider the two minimax
levels
\[
\overline\lambda
=
\sup_{u\in\mathcal S\setminus\{0\}}
\inf_{v\in\mathcal S^o}
\mathcal R(u,v),
\qquad
\underline\lambda
=
\inf_{v\in\mathcal S\setminus\{0\}}
\sup_{u\in\mathcal S^o}
\mathcal R(u,v).
\]
When these levels coincide, their common value is denoted by
\(\hat\lambda\). The main questions are whether this value is related to
eigenvalues of the pencil, whether it can be recovered by finite-dimensional
minimax approximation, and whether it gives computable one-sided certificates.

If a positive right-left eigenpair is already known, then the corresponding
eigenvalue is naturally detected by the two-variable quotient. The point of
the present work is different: we use this quotient as a variational object
in its own right and study the real cone levels selected by the associated
minimax construction. In finite dimensions, related cone minimax formulas for
matrices were studied in~\cite{IlyasovValeev2024}. The present paper develops
this viewpoint for operator pencils and for their finite-dimensional
approximations.

In the abstract part we prove that every cone quasi-eigenvalue is trapped
between the two minimax levels. If the relevant extremal relations are
attained at strictly positive vectors, then the two levels collapse and the
common value admits a saddle-point interpretation. In particular, whenever a
strictly positive right-left eigenpair exists, the associated eigenvalue
realizes the cone minimax identity.

A central feature of the minimax identity is its a posteriori content. This
is not merely an additional consequence, but one of the main reasons for
working with the two-variable quotient. Once
\[
\overline\lambda=\underline\lambda=\hat\lambda,
\]
every admissible trial element gives a one-sided bound for the selected
spectral value. Namely,
\[
\lambda_R(u)
:=
\inf_{v\in\mathcal S^o}\mathcal R(u,v),
\qquad
\lambda_L(v)
:=
\sup_{u\in\mathcal S^o}\mathcal R(u,v)
\]
satisfy
\[
\lambda_R(u)\le \hat\lambda\le \lambda_L(v).
\]
Thus the minimax formula is not only a characterization theorem; it also gives
directly verifiable lower and upper certificates. The results below identify
conditions under which these continuous certificates are approximated by
finite-dimensional, componentwise computable Collatz--Wielandt certificates,
with a primal--dual gap providing a practical stopping criterion. This
distinguishes the method from a purely existential spectral theory.

The main approximation results concern elliptic operator pencils. We use
standard conforming finite-dimensional approximations in the spirit of the
finite element method~\cite{ciaret,ciaretRav2,ciaretRav,Ern}. Under
finite-dimensional mixed positivity and boundary-exclusion assumptions, the
discrete minimax levels collapse to positive right-left eigenpairs. Under
compactness and uniform boundedness assumptions, subsequences of these levels
have limits, and the corresponding extremal sequences converge to limiting
variational objects. Under additional strong cone assumptions on
\(\mathcal L\) and \(\mathcal L^*\), these limits are identified with
eigenpairs of the infinite-dimensional pencil, and the limiting value is
identified with the continuous cone minimax level.

The maximum-principle input in this part has a specific role. It is not used
to define the principal value or to reduce the pencil to a positive resolvent.
Rather, after the Galerkin minimax limit has produced nonnegative right-left
quasi-eigenvectors, a cone form of the strong maximum principle is used to
upgrade them to strictly positive right-left eigenvectors.

We also prove one-sided perturbation and approximation estimates. These
estimates complement classical perturbation theory for isolated eigenvalues
\cite{Kato}: here the selected value is controlled by cone minimax quantities
rather than by a single eigenvalue derivative. Simple model examples show
that the estimates are sharp.

The paper includes examples illustrating the scope of the method. One model
treats a genuinely coupled non-selfadjoint elliptic pencil with singular
weight, for which the hypotheses of the convergence theorem can be verified
explicitly but no closed formula for the minimax value is available. Another
example computes the cone minimax value for a two-component elliptic system
with a fixed-sign skew-symmetric coupling. This system is non-cooperative with
respect to the standard componentwise order, and the coupling does not
preserve the positive cone. Nevertheless, the cone minimax value is explicit
and is determined by the principal eigenvalue of the underlying scalar drift
operator.

The paper is organized as follows. Section~\ref{sec:abstract} introduces the
abstract cone minimax framework, cone quasi-eigenvalues, and the trapping
principle. Section~\ref{sec:approximation} develops the discrete minimax
formula and the finite-dimensional minimax collapse. Section~\ref{sec:left-minimax-attainment}
proves attainment of the lower principal level. Section~\ref{sec:pde} passes
to the Galerkin limit and identifies the continuous minimax value under strong
cone assumptions; it also discusses the relation with shifted Krein--Rutman
reductions and maximum-principle arguments. Section~\ref{sec:aposteriori-enclosures}
gives a posteriori spectral enclosures and primal--dual gaps.
Section~\ref{sec:stability} proves one-sided perturbation and
finite-dimensional approximation error estimates.
Section~\ref{sec:complete-singular-weight-model} presents a genuinely coupled
non-selfadjoint model with singular weight. Section~\ref{sec:exampl} treats a
non-cooperative skew-coupled elliptic system. The final section contains
concluding remarks.

\section{Abstract minimax framework}
\label{sec:abstract}

We develop the minimax mechanism in an abstract ordered setting. The
structure is deliberately minimal at this stage: the essential ingredients
are the cone, the positivity of the denominator, and the two-sided Rayleigh
quotient. Compactness and the specific elliptic form of the operators enter
only later, when the abstract hypotheses are verified in concrete
applications.

\medskip

\noindent\textbf{Functional-analytic setting.}\;
Throughout, $\mathcal{V}$ is a real Banach space and $H$ is a real Hilbert
space forming a Gelfand triple
\[
\mathcal{V} \hookrightarrow H \equiv H^{*} \hookrightarrow \mathcal{V}^{*}
\]
with dense continuous embeddings; the duality pairing is written
$\langle\cdot,\cdot\rangle$. We fix an ordered Banach function space
$\mathcal{C}^{0}$, and all order relations are understood in $\mathcal{C}^{0}$.
When a limiting argument requires it, we shall additionally assume that
$\mathcal{V}$ is reflexive and that the embedding $\mathcal{V}\hookrightarrow H$
is compact.

\medskip

\noindent\textbf{The cone.}\;
The central object is a convex cone
$\mathcal{S}^{o}\subset\mathcal{V}\cap\mathcal{C}^{0}$ of strictly positive
elements. Its closure in $\mathcal{C}^{0}$, intersected with $\mathcal{V}$,
gives the full cone
\[
\mathcal{S} := \overline{\mathcal{S}^{o}}^{\,\mathcal{C}^{0}}\cap\mathcal{V}.
\]
We assume that $\mathcal{S}$ is pointed and that $\mathcal{S}^{o}$ is dense
in $\mathcal{S}\setminus\{0\}$ in the $\mathcal{V}$-topology. In applications
$\mathcal{S}^{o}$ is typically the strictly positive part of the cone, though
no topological interior assumption is imposed here.

\medskip

\noindent\textbf{The spectral problem.}\;
Let $\mathcal{L}:\mathcal{V}\to\mathcal{V}^{*}$ and $\mathcal{G}:H\to H$ be
bounded linear operators. We impose the following positivity condition on the
weight.

\begin{description}
	\item[$(P)$] \emph{Strict positivity of the weight.}
	\[
	(\mathcal{G}u,v)_H>0
	\qquad
	\forall\,(u,v)\in\mathcal D,
	\]
	where
	\[
	\mathcal D
	:=
	(\mathcal S\setminus\{0\})\times\mathcal S^o
	\cup
	\mathcal S^o\times(\mathcal S\setminus\{0\}).
	\]
\end{description}

We regard $\mathcal G u\in H$ as an element of $\mathcal V^*$ through the
embedding $H\hookrightarrow\mathcal V^*$, and study the generalized
eigenvalue problem
\[
\mathcal{L}u = \lambda\mathcal{G}u.
\]

\medskip

\noindent\textbf{The Rayleigh quotient and the minimax levels.}\;
Whenever condition $(P)$ is in force, the denominator is positive and the
\emph{extended Rayleigh quotient}
\[
\mathcal{R}(u,v)
:= \frac{\langle\mathcal{L}u,v\rangle}{(\mathcal{G}u,v)_{H}}
\]
is well defined. For a fixed trial vector $u\in\mathcal{S}\setminus\{0\}$,
optimizing over positive test vectors gives the one-sided quantity
\[
\overline{\lambda}(u)
:= \inf_{w\in\mathcal{S}^{o}}\mathcal{R}(u,w),
\]
and fixing instead the test direction $v\in\mathcal{S}\setminus\{0\}$ gives
\[
\underline{\lambda}(v)
:= \sup_{w\in\mathcal{S}^{o}}\mathcal{R}(w,v).
\]
The two \emph{principal minimax levels} are the global extrema of these
quantities over the cone:
\begin{align*}
	\overline{\lambda}
	&:= \sup_{\phi\in\mathcal{S}\setminus\{0\}}\overline{\lambda}(\phi)
	= \sup_{\phi\in\mathcal{S}\setminus\{0\}}
	\inf_{w\in\mathcal{S}^{o}}\mathcal{R}(\phi,w),\\
	\underline{\lambda}
	&:= \inf_{\psi\in\mathcal{S}\setminus\{0\}}\underline{\lambda}(\psi)
	= \inf_{\psi\in\mathcal{S}\setminus\{0\}}
	\sup_{w\in\mathcal{S}^{o}}\mathcal{R}(w,\psi).
\end{align*}
The notation is mnemonic rather than order-theoretic:
$\overline\lambda$ denotes the right sup--inf level, while
$\underline\lambda$ denotes the left inf--sup level. No ordering between
$\overline{\lambda}$ and $\underline{\lambda}$ is assumed; identifying
conditions under which they coincide is the main task of this section.

\medskip

\noindent\textbf{Quasi-eigenvectors and quasi-eigenvalues.}\;
A quasi-eigenvector is not defined first by an algebraic eigenvalue equation,
but by realization of one of the cone minimax profiles. We call a vector
$\phi_{*}\in\mathcal{S}\setminus\{0\}$ a \emph{right cone quasi-eigenvector}
at level $\lambda_{*}$ if
\[
\lambda_{*}
=
\overline{\lambda}(\phi_{*})
=
\inf_{w\in\mathcal{S}^{o}}\mathcal{R}(\phi_{*},w);
\]
a vector $\psi_{*}\in\mathcal{S}\setminus\{0\}$ is a
\emph{left cone quasi-eigenvector} at level $\lambda_{*}$ if
\[
\lambda_{*}
=
\underline{\lambda}(\psi_{*})
=
\sup_{w\in\mathcal{S}^{o}}\mathcal{R}(w,\psi_{*}).
\]
A real number $\lambda_{*}$ is a \emph{cone quasi-eigenvalue} if it is
simultaneously realized by a right and a left quasi-eigenvector. If
\[
\hat\lambda
:=
\overline{\lambda}
=
\underline{\lambda},
\]
namely
\[
\hat\lambda
=
\sup_{\phi\in\mathcal{S}\setminus\{0\}}
\inf_{w\in\mathcal{S}^{o}}\mathcal{R}(\phi,w)
=
\inf_{\psi\in\mathcal{S}\setminus\{0\}}
\sup_{w\in\mathcal{S}^{o}}\mathcal{R}(w,\psi),
\]
we say that the \emph{minimax condition} holds, and a realizing pair
$(\phi_{*},\psi_{*})$ is called a \emph{principal right-left quasi-pair}.

Every strictly positive right-left eigenpair---vectors
$\phi_{*},\psi_{*}\in\mathcal{S}^{o}$ satisfying
$\mathcal{L}\phi_{*}=\lambda_{*}\mathcal{G}\phi_{*}$ and
$\mathcal{L}^{*}\psi_{*}=\lambda_{*}\mathcal{G}^{*}\psi_{*}$---automatically
satisfies $\overline{\lambda}(\phi_{*})=\lambda_{*}=\underline{\lambda}(\psi_{*})$
and is therefore a quasi-pair. The converse need not hold: the quasi-eigenvalue
notion may produce a meaningful real level even when no conventional real
eigenpair exists.

\medskip

\begin{example}[A real minimax level from non-real spectrum]
	Take
	\[
	A = \begin{pmatrix}\alpha & \omega \\ -\omega & \alpha\end{pmatrix},
	\qquad G = I,
	\qquad \mathcal{S} = \mathbb{R}_{+}^{2},
	\qquad \mathcal{S}^{o} = (0,\infty)^{2}.
	\]
	The spectrum of $A$ is $\{\alpha\pm i\omega\}$, which is non-real whenever
	$\omega\ne0$. Nevertheless a direct computation gives
	$\underline{\lambda}=\overline{\lambda}=\alpha$, with quasi-pair $(e_{2},e_{2})$
	for $\omega>0$ and $(e_{1},e_{1})$ for $\omega<0$. Indeed, the skew part can
	be made one-sidedly invisible by testing along the corresponding boundary
	direction of the cone. The minimax construction thus extracts a real cone
	level from an operator with no real eigenvalue at that level; the
	quasi-eigenvectors play the role that conventional eigenvectors cannot.
\end{example}

\begin{lemma}[Minimax trapping]\label{lem:main}
	Every cone quasi-eigenvalue $\lambda_{*}$ satisfies
	\[
	\underline{\lambda} \leq \lambda_{*} \leq \overline{\lambda}.
	\]
	In particular, if the minimax condition holds, then $\hat\lambda$ is the
	unique cone quasi-eigenvalue.
\end{lemma}

\begin{proof}
	If $\lambda_{*}$ is realized by $\phi_{*},\psi_{*}\in\mathcal{S}\setminus\{0\}$,
	then $\overline{\lambda}(\phi_{*})=\lambda_{*}\leq\overline{\lambda}$ and
	$\underline{\lambda}\leq\underline{\lambda}(\psi_{*})=\lambda_{*}$.
\end{proof}

The next result shows that when the realizing vectors lie strictly inside the
cone, the trapping is sharp.

\begin{corollary}[Interior quasi-eigenvectors collapse the minimax levels]
	\label{cor:interior-eigenvalue-collapse}
	If $\lambda_{*}$ is a cone quasi-eigenvalue realized by
	$\phi_{*},\psi_{*}\in\mathcal{S}^{o}$, then
	$\underline{\lambda}=\lambda_{*}=\overline{\lambda}$.
	Consequently, failure of the minimax condition implies that no real eigenvalue
	admits a strictly positive right-left eigenpair.
\end{corollary}

\begin{proof}
	By Lemma~\ref{lem:main} it suffices to establish the reverse inequalities
	$\underline{\lambda}\geq\lambda_{*}$ and $\overline{\lambda}\leq\lambda_{*}$.
	We prove the first; the second is symmetric.
	
	Take any $\psi\in\mathcal{S}\setminus\{0\}$ and choose
	$\psi_{n}\in\mathcal{S}^{o}$ with $\psi_{n}\to\psi$ in $\mathcal{V}$.
	Since $\overline{\lambda}(\phi_{*})=\lambda_{*}$, we have
	$\mathcal{R}(\phi_{*},\psi_{n})\geq\lambda_{*}$ for all $n$. As
	$\phi_{*}\in\mathcal{S}^{o}$, condition $(P)$ gives
	$(\mathcal{G}\phi_{*},\psi)_{H}>0$, so the quotient is continuous in
	$\psi_{n}$ and the inequality survives in the limit:
	$\mathcal{R}(\phi_{*},\psi)\geq\lambda_{*}$. Since
	$\phi_{*}\in\mathcal{S}^{o}$, this means
	$\underline{\lambda}(\psi)\geq\mathcal{R}(\phi_{*},\psi)\geq\lambda_{*}$.
	As $\psi$ was arbitrary, $\underline{\lambda}\geq\lambda_{*}$.
	The proof of $\overline{\lambda}\le\lambda_{*}$ is the same, using
	$\psi_{*}\in\mathcal S^o$ and the identity
	$\underline{\lambda}(\psi_{*})=\lambda_{*}$.
\end{proof}

\begin{remark}[Strict positivity is essential]
	The conclusion fails for boundary eigenvectors. For
	$A=\operatorname{diag}(1,2)$, $G=I$, $\mathcal{S}=\mathbb{R}_{+}^{2}$, one
	finds $\underline{\lambda}=1<2=\overline{\lambda}$. The eigenvalue $2$ is
	realized by $e_{2}\in\partial\mathcal{S}$, yet the levels do not collapse,
	because $e_{2}$ is not an admissible strictly positive test direction.
\end{remark}

A complementary collapse mechanism operates even when the realizing vectors
reach the boundary of the cone, provided both principal levels are attained.

\begin{lemma}[Saddle collapse]\label{lem:saddle-collapse}
	Suppose $\overline{\lambda}$ and $\underline{\lambda}$ are finite and attained
	at $\phi^{*},\psi^{*}\in\mathcal{S}\setminus\{0\}$, with
	$(\mathcal{G}\phi^{*},\psi^{*})_{H}>0$. Then
	\[
	\overline{\lambda}
	\leq \mathcal{R}(\phi^{*},\psi^{*})
	\leq \underline{\lambda}.
	\]
	If, additionally, $\underline{\lambda}\leq\overline{\lambda}$ then
	\[
	\underline{\lambda}
	= \overline{\lambda}
	= \mathcal{R}(\phi^{*},\psi^{*}),
	\]
	and $(\phi^{*},\psi^{*})$ is a saddle point:
	\[
	\mathcal{R}(u,\psi^{*})
	\leq \mathcal{R}(\phi^{*},\psi^{*}) \leq
	\mathcal{R}(\phi^{*},v)
	\qquad \forall\,u,v\in\mathcal{S}^{o}.
	\]
\end{lemma}

\begin{proof}
	From $\overline{\lambda}(\phi^{*})=\overline{\lambda}$ we have
	$\mathcal{R}(\phi^{*},v)\geq\overline{\lambda}$ for all $v\in\mathcal{S}^{o}$.
	Choose $v_{n}\in\mathcal{S}^{o}$ with $v_{n}\to\psi^{*}$ in $\mathcal{V}$;
	the assumption $(\mathcal{G}\phi^{*},\psi^{*})_{H}>0$ allows passage to the
	limit, giving $\mathcal{R}(\phi^{*},\psi^{*})\geq\overline{\lambda}$.
	The symmetric argument using $\underline{\lambda}(\psi^{*})=\underline{\lambda}$
	yields $\mathcal{R}(\phi^{*},\psi^{*})\leq\underline{\lambda}$.
	If $\underline{\lambda}\leq\overline{\lambda}$, the three quantities are
	squeezed to equality, and the saddle inequalities follow directly from the
	definitions of $\overline{\lambda}(\phi^{*})$ and
	$\underline{\lambda}(\psi^{*})$. The additional ordering
	$\underline{\lambda}\leq\overline{\lambda}$ holds, for instance, whenever a
	cone quasi-eigenvalue exists, by Lemma~\ref{lem:main}.
\end{proof}

\begin{remark}[Relation to the spectrum]
	The cone minimax construction depends on the ordered triple
	$(\mathcal{L},\mathcal{G},\mathcal{S})$, not on $\mathcal{L}$ alone. When a
	strictly positive right-left eigenpair exists for a real eigenvalue
	$\lambda_{*}$, Corollary~\ref{cor:interior-eigenvalue-collapse} forces the
	minimax condition and identifies $\hat\lambda$ with $\lambda_{*}$.
	Contrapositively, failure of the minimax condition certifies the absence of
	any such eigenpair.
	
	When the relevant spectrum is non-real, as in the rotation example, the
	construction may still yield a real level $\hat\lambda$, but the corresponding
	quasi-pair is not a genuine spectral eigenpair; it is an intrinsically
	cone-theoretic object with no classical counterpart.
\end{remark}


\section{Discrete minimax formula}
\label{sec:approximation}

A notable feature of the finite-dimensional part is its generality. The
minimax identities below use only the ordered finite-dimensional cone
structure, strict positivity of the Galerkin weight, and a boundary-exclusion
condition. No ellipticity, cone invariance of the operator, resolvent
positivity, or compactness assumption is needed at this stage. These
additional analytic assumptions enter only later, when the discrete
minimax identities are connected with Galerkin limits of infinite-dimensional
operator pencils.

\medskip

\noindent\textbf{Elliptic realization.}\;
From this point on, when Galerkin approximation and elliptic examples are
considered, we use the following concrete realization of the abstract setting.
Let \(\Omega\subset\mathbb R^n\) be a bounded domain, set
\[
V:=H_0^1(\Omega),
\qquad
\mathcal V:=V^m,
\qquad
H:=[L^2(\Omega)]^m,
\]
and consider
\begin{equation}\label{eq:L_system}
	\mathcal L\mathbf u
	=
	-\sum_{i,j=1}^n \partial_j\!\bigl(A_{ij}(x)\partial_i\mathbf u\bigr)
	+
	\sum_{i=1}^n B_i(x)\partial_i\mathbf u
	+
	C(x)\mathbf u,
	\qquad
	\mathcal G\mathbf u=G(x)\mathbf u ,
\end{equation}
where
\[
A_{ij},\,B_i,\,C,\,G\in L^\infty(\Omega)^{m\times m}.
\]
No symmetry assumptions are imposed on these matrices. The operator
\(\mathcal L\) is understood through its weak realization
\(\mathcal L:\mathcal V\to\mathcal V^*\), given by
\[
\begin{aligned}
	\langle \mathcal L\mathbf u,\mathbf v\rangle
	&=
	\sum_{i,j=1}^n
	\int_\Omega
	\bigl(A_{ij}(x)\partial_i\mathbf u\bigr)\cdot
	\partial_j\mathbf v\,dx
	+
	\sum_{i=1}^n
	\int_\Omega
	\bigl(B_i(x)\partial_i\mathbf u\bigr)\cdot
	\mathbf v\,dx                                      \\
	&\quad
	+
	\int_\Omega
	\bigl(C(x)\mathbf u\bigr)\cdot \mathbf v\,dx,
	\qquad \mathbf u,\mathbf v\in\mathcal V .
\end{aligned}
\]
For admissible pairs we keep the notation
\begin{equation}\label{eq:Rel}
	\mathcal R(\mathbf u,\mathbf v)
	:=
	\frac{\langle \mathcal L\mathbf u,\mathbf v\rangle}
	{(\mathcal G\mathbf u,\mathbf v)_H}.
\end{equation}

We assume that $\Omega$ admits an interior polyhedral exhaustion: there exist
polyhedral subdomains $\Omega_r\subset\Omega$ with $\bigcup_{r\ge1}\Omega_r=\Omega$.
The index $r$ encodes both the choice of subdomain and the mesh size. On each
$\Omega_r$ we fix a shape-regular conforming $P_1$-triangulation $\mathcal{T}_r$
and let $V_r$ be the corresponding finite element space with zero boundary values
on $\partial\Omega_r$, extended by zero to $\Omega$, so that $V_r\subset V$. We
assume the standard Galerkin density property
\begin{equation}\label{eq:Wr-density}
	\forall\,\xi\in V \quad \exists\,\xi_r\in V_r \quad\text{such that}\quad
	\xi_r\longrightarrow\xi \quad\text{strongly in }V,
\end{equation}
which holds for standard interior polyhedral exhaustions and shape-regular
conforming spaces; see~\cite{ciaretRav,ciaret,Ern}.

Let $B_1,\ldots,B_{N_r}$ be the interior nodes of $\mathcal{T}_r$ with nodal
basis $\{\psi_1,\ldots,\psi_{N_r}\}$ satisfying $\psi_i(B_j)=\delta_{ij}$,
$\psi_i\ge0$, and $\psi_i=0$ on $\partial\Omega_r$. After extension by zero,
each $\psi_i$ belongs to $V\cap C(\overline\Omega)$.

\medskip

\noindent\textbf{Discrete cones and operators.}\;
Define the scalar nodal cones
\begin{align*}
	&S_r^o
	:= \Bigl\{\,u_r=\textstyle\sum_i u_i\psi_i\in V_r : u_i>0\Bigr\},
	\\
	&S_r
	:= \overline{S_r^o}^{\,V_r}
	= \Bigl\{\,u_r=\textstyle\sum_i u_i\psi_i\in V_r : u_i\ge0\Bigr\}.
\end{align*}
Thus $S_r^o$ is an open cone in $V_r$ and $S_r$ is a closed polyhedral cone
with basis $\{\psi_1,\ldots,\psi_{N_r}\}$. For systems we set
\[
\mathcal{V}_r := V_r^m, \qquad
\mathcal{S}_r^o := (S_r^o)^m, \qquad
\mathcal{S}_r := (S_r)^m,
\]
with basis vectors $\Phi_{i,k}:=\psi_i e_k$, $i=1,\ldots,N_r$,
$k=1,\ldots,m$. Since every element of $\mathcal{S}_r$ is continuous,
componentwise nonnegative, and belongs to $\mathcal{V}$, we have
$\mathcal{S}_r\subset\mathcal{S}$.

\medskip

For $\mathbf{u}_r=\sum_{i,k}u_k^i\Phi_{i,k}\in\mathcal{V}_r$, write
$\bar{\mathbf{u}}:=(u_k^i)_{i,k}\in\mathbb{R}^{mN_r}$ for its coefficient
vector, with all inequalities understood componentwise. Then
\[
\mathbf{u}_r\in\mathcal{S}_r \Longleftrightarrow \bar{\mathbf{u}}\ge0,
\qquad
\mathbf{u}_r\in\mathcal{S}_r^o \Longleftrightarrow \bar{\mathbf{u}}>0.
\]

For $\mathbf{u}\in\mathcal{C}^0$, the componentwise nodal interpolant is
\[
\mathcal{I}_r\mathbf{u}
:= \sum_{i=1}^{N_r}\sum_{k=1}^m u_k(B_i)\Phi_{i,k}.
\]
Since every node $B_i$ lies in $\Omega$, interpolation preserves cone membership:
\begin{equation}\label{eq:positive-interpolation}
	\mathcal{I}_r(\mathcal{S})\subset\mathcal{S}_r,
	\qquad
	\mathcal{I}_r(\mathcal{S}^o)\subset\mathcal{S}_r^o.
\end{equation}

Define the stiffness and weight matrices by
\[
(L_r)_{(j,\ell),(i,k)}
:= \langle\mathcal{L}\Phi_{i,k},\Phi_{j,\ell}\rangle,
\qquad
(G_r)_{(j,\ell),(i,k)}
:= (\mathcal{G}\Phi_{i,k},\Phi_{j,\ell})_H,
\]
so that $\bar{\mathbf{v}}^T L_r\bar{\mathbf{u}}=\langle\mathcal{L}\mathbf{u}_r,
\mathbf{v}_r\rangle$ and $\bar{\mathbf{v}}^T G_r\bar{\mathbf{u}}=
(\mathcal{G}\mathbf{u}_r,\mathbf{v}_r)_H$. Whenever the denominator is positive,
the discrete extended Rayleigh quotient
\begin{equation}\label{eq:discrete-Rayleigh}
	\mathcal{R}_r(\mathbf{u}_r,\mathbf{v}_r)
	:= \frac{\bar{\mathbf{v}}^T L_r\bar{\mathbf{u}}}
	{\bar{\mathbf{v}}^T G_r\bar{\mathbf{u}}}
\end{equation}
is the restriction of $\mathcal{R}$ to $\mathcal{V}_r\times\mathcal{V}_r$.

\subsection{Minimax characterization}

We work under the following two assumptions.

\begin{description}
	\item[$(P_r)$] \emph{Strict positivity of the Galerkin weight.}
	\[
	\bar{\mathbf{v}}^T G_r\bar{\mathbf{u}}>0
	\qquad
	\forall\,(\mathbf{u}_r,\mathbf{v}_r)\in\mathcal{D}_r,
	\]
	where $\mathcal{D}_r:=(\mathcal{S}_r\setminus\{\mathbf{0}\})\times\mathcal{S}_r^o
	\cup\,\mathcal{S}_r^o\times(\mathcal{S}_r\setminus\{\mathbf{0}\})$.
	
	\item[$(B_r)$] \emph{Boundary exclusion.} For $M=L_r,L_r^T$,
	\[
	\bar{\mathbf{w}}\ge0,\quad
	\bar{\mathbf{w}}\ne0,\quad
	M\bar{\mathbf{w}}\ge0
	\quad\Longrightarrow\quad
	\bar{\mathbf{w}}>0.
	\]
\end{description}

Define the four mixed minimax levels
\begin{align}
	\overline\lambda_r
	&:= \sup_{\mathbf{u}\in\mathcal{S}_r\setminus\{\mathbf{0}\}}
	\inf_{\mathbf{v}\in\mathcal{S}_r^o}
	\mathcal{R}_r(\mathbf{u},\mathbf{v}),
	&
	\overline\Lambda_r
	&:= \inf_{\mathbf{v}\in\mathcal{S}_r^o}
	\sup_{\mathbf{u}\in\mathcal{S}_r\setminus\{\mathbf{0}\}}
	\mathcal{R}_r(\mathbf{u},\mathbf{v}),
	\label{eq:discrete-upper-minimax}\\
	\underline\Lambda_r
	&:= \sup_{\mathbf{u}\in\mathcal{S}_r^o}
	\inf_{\mathbf{v}\in\mathcal{S}_r\setminus\{\mathbf{0}\}}
	\mathcal{R}_r(\mathbf{u},\mathbf{v}),
	&
	\underline\lambda_r
	&:= \inf_{\mathbf{v}\in\mathcal{S}_r\setminus\{\mathbf{0}\}}
	\sup_{\mathbf{u}\in\mathcal{S}_r^o}
	\mathcal{R}_r(\mathbf{u},\mathbf{v}).
	\label{eq:discrete-lower-minimax}
\end{align}

\begin{lemma}[Discrete minimax principle]\label{lem:discrete-minimax}
	Assume $(P_r)$. Then both pairs of levels coincide and are finite:
	\[
	-\infty < \overline\lambda_r = \overline\Lambda_r < +\infty,
	\qquad
	-\infty < \underline\lambda_r = \underline\Lambda_r < +\infty.
	\]
	Moreover, there exist $\mathbf{u}_r,\mathbf{v}_r\in\mathcal{S}_r\setminus\{\mathbf{0}\}$
	such that
	\begin{equation}\label{eq:minmaxD2}
		\overline\lambda_r
		= \inf_{\mathbf{v}\in\mathcal{S}_r^o}\mathcal{R}_r(\mathbf{u}_r,\mathbf{v}),
		\qquad
		\underline\lambda_r
		= \sup_{\mathbf{u}\in\mathcal{S}_r^o}\mathcal{R}_r(\mathbf{u},\mathbf{v}_r).
	\end{equation}
\end{lemma}

\begin{proof}
	By homogeneity, all variables may be normalized to the compact convex set
	\[
	K_r := \Bigl\{\mathbf{u}\in\mathcal{S}_r :
	\textstyle\sum_{i,k}u_k^i=1\Bigr\},
	\qquad K_r^o := K_r\cap\mathcal{S}_r^o,
	\]
	where $K_r^o$ is convex and dense in $K_r$. By $(P_r)$, the denominator of
	$\mathcal{R}_r$ is positive on $K_r\times K_r^o$ and $K_r^o\times K_r$.
	
	The quotient is continuous and linear-fractional in each variable;
	hence it is quasiconcave in its first variable and quasiconvex in its
	second. Sion's minimax theorem~\cite{Sion1958} therefore gives
	\[
	\overline\lambda_r=\overline\Lambda_r.
	\]
	Applying it to
	\((\mathbf v,\mathbf u)\mapsto-\mathcal R_r(\mathbf u,\mathbf v)\)
	gives
	\[
	\underline\lambda_r=\underline\Lambda_r.
	\]
	Finiteness follows because, for any fixed
	$\mathbf{u}^0,\mathbf{v}^0\in K_r^o$, the denominators
	$\bar{\mathbf{v}}^T G_r\bar{\mathbf{u}}^0$ and
	$(\bar{\mathbf{v}}^0)^T G_r\bar{\mathbf{u}}$ are bounded away from zero on
	$K_r$. Finally, the maps
	$\mathbf{u}\mapsto\inf_{\mathbf{v}\in K_r^o}\mathcal{R}_r(\mathbf{u},\mathbf{v})$
	and $\mathbf{v}\mapsto\sup_{\mathbf{u}\in K_r^o}\mathcal{R}_r(\mathbf{u},\mathbf{v})$
	are upper and lower semicontinuous on the compact set $K_r$, respectively, so
	their extrema are attained, proving~\eqref{eq:minmaxD2}.
\end{proof}

\begin{lemma}[Discrete minimax collapse]\label{lem:discrete-collapse}
	Assume $(P_r)$ and $(B_r)$, and let $\mathbf{u}_r,\mathbf{v}_r$ be as in
	Lemma~\ref{lem:discrete-minimax}. If $\overline\lambda_r\ge0$, then
	$\mathbf{u}_r,\mathbf{v}_r\in\mathcal{S}_r^o$,
	\[
	\hat\lambda_r := \overline\lambda_r = \underline\lambda_r,
	\qquad
	\mathcal{R}_r(\mathbf{u}_r,\mathbf{v}_r) = \hat\lambda_r,
	\]
	and the pair $(\mathbf{u}_r,\mathbf{v}_r)$ satisfies the discrete eigenvalue
	equations
	\begin{equation}\label{eq:uravn}
		L_r\bar{\mathbf{u}}_r = \hat\lambda_r G_r\bar{\mathbf{u}}_r,
		\qquad
		L_r^T\bar{\mathbf{v}}_r = \hat\lambda_r G_r^T\bar{\mathbf{v}}_r.
	\end{equation}
\end{lemma}

\begin{proof}
	We first prove that $\mathbf{u}_r$ is strictly positive. By
	\eqref{eq:minmaxD2},
	$\mathcal{R}_r(\mathbf{u}_r,\mathbf{v})\ge\overline\lambda_r$
	for all $\mathbf{v}\in\mathcal{S}_r^o$. Since the denominator is positive,
	\[
	\bar{\mathbf v}^{T}(L_r-\overline\lambda_r G_r)\bar{\mathbf u}_r\ge0
	\qquad \forall\,\mathbf v\in\mathcal S_r^o,
	\]
	and hence $(L_r-\overline\lambda_r G_r)\bar{\mathbf{u}}_r\ge0$ componentwise.
	Also, $(P_r)$ gives $G_r\bar{\mathbf{u}}_r\ge0$: indeed,
	$\bar{\mathbf v}^{T}G_r\bar{\mathbf u}_r>0$ for all
	$\mathbf v\in\mathcal S_r^o$, and passage to boundary directions yields
	componentwise nonnegativity. Since $\overline\lambda_r\ge0$, it follows that
	$L_r\bar{\mathbf{u}}_r\ge0$. Therefore $(B_r)$ gives
	$\bar{\mathbf{u}}_r>0$, that is, $\mathbf{u}_r\in\mathcal{S}_r^o$.
	
	Since $\mathbf{u}_r\in\mathcal{S}_r^o$, $(P_r)$ keeps the denominator positive
	for $\mathbf v\in\mathcal S_r\setminus\{\mathbf0\}$, and the infimum in
	\eqref{eq:minmaxD2} extends by continuity from $\mathcal{S}_r^o$ to
	$\mathcal{S}_r\setminus\{\mathbf{0}\}$. Thus
	$\underline\Lambda_r\ge\overline\lambda_r$. Conversely, for every
	$\mathbf u\in\mathcal{S}_r^o$,
	\[
	\inf_{\mathbf v\in\mathcal{S}_r\setminus\{\mathbf0\}}
	\mathcal R_r(\mathbf u,\mathbf v)
	\le
	\inf_{\mathbf v\in\mathcal{S}_r^o}
	\mathcal R_r(\mathbf u,\mathbf v)
	\le
	\overline\lambda_r,
	\]
	and therefore $\underline\Lambda_r\le\overline\lambda_r$. Hence
	\[
	\underline\lambda_r=\underline\Lambda_r=\overline\lambda_r=\overline\Lambda_r=:\hat\lambda_r.
	\]
	
	It remains to identify the saddle value and prove strict positivity of
	$\mathbf{v}_r$. From \eqref{eq:minmaxD2},
	$\mathcal{R}_r(\mathbf{u},\mathbf{v}_r)\le\hat\lambda_r$ for all
	$\mathbf{u}\in\mathcal{S}_r^o$, while the extended infimum gives
	$\mathcal{R}_r(\mathbf{u}_r,\mathbf{v}_r)\ge\hat\lambda_r$. Since
	$\mathbf{u}_r\in\mathcal{S}_r^o$, the first inequality applies to
	$\mathbf{u}_r$ as well, and therefore
	\[
	\mathcal{R}_r(\mathbf{u}_r,\mathbf{v}_r)=\hat\lambda_r.
	\]
	
	Now $\mathcal{R}_r(\mathbf{u},\mathbf{v}_r)\le\hat\lambda_r$ for all
	$\mathbf u\in\mathcal S_r^o$, with equality at the strictly positive point
	$\mathbf u_r$. Hence, for every coefficient direction $\boldsymbol\xi$ and
	all sufficiently small \(t\), the vector
	\(\bar{\mathbf u}_r+t\boldsymbol\xi\) remains positive, and differentiation
	at \(t=0\) gives
	\[
	L_r^T\bar{\mathbf{v}}_r=\hat\lambda_r G_r^T\bar{\mathbf{v}}_r.
	\]
	By $(P_r)$, $G_r^T\bar{\mathbf{v}}_r\ge0$: indeed,
	$\bar{\mathbf v}_r^T G_r\bar{\mathbf u}>0$ for all
	$\mathbf u\in\mathcal S_r^o$, and passage to boundary directions yields
	componentwise nonnegativity. Since $\hat\lambda_r\ge0$, this implies
	$L_r^T\bar{\mathbf{v}}_r\ge0$. Applying $(B_r)$ to $L_r^T$ yields
	$\bar{\mathbf{v}}_r>0$, that is, $\mathbf{v}_r\in\mathcal{S}_r^o$.
	
	Finally, since $\mathbf{v}_r$ is also strictly positive and
	$\mathcal{R}_r(\mathbf{u}_r,\mathbf{v})\ge\hat\lambda_r$ for all
	$\mathbf v\in\mathcal S_r^o$, with equality at $\mathbf v_r$, differentiating
	with respect to arbitrary coefficient directions gives
	\[
	L_r\bar{\mathbf{u}}_r=\hat\lambda_r G_r\bar{\mathbf{u}}_r.
	\]
	This proves \eqref{eq:uravn}.
\end{proof}

\section{Existence of a left cone quasi-eigenvector}
\label{sec:left-minimax-attainment}

We first treat the lower/left level separately, because this part of the
construction admits a compactness argument under weaker assumptions and will
be used later in identifying the Galerkin limit. Throughout this section we
use the standing compactness assumptions that $\mathcal V$ is reflexive, the
embedding $\mathcal V\hookrightarrow H$ is compact, and $\mathcal S$ is weakly
closed in $\mathcal V$.

We assume the uniform ellipticity condition: there exists $\theta>0$ such that,
for a.e.\ $x\in\Omega$,
\begin{equation}\label{eq:unifellipt}
	\sum_{i,j=1}^n
	\langle A_{ij}(x)\xi_i,\xi_j\rangle_{\mathbb{R}^m}
	\ge \theta|\xi|^2
	\qquad
	\forall\,\xi=(\xi_1,\ldots,\xi_n)\in(\mathbb{R}^m)^n.
\end{equation}
This implies the G{\aa}rding inequality: there exist $c>0$ and $C\ge0$ such that
\[
\langle\mathcal{L}\mathbf{u},\mathbf{u}\rangle
\ge c\|\mathbf{u}\|_{\mathcal{V}}^2 - C\|\mathbf{u}\|_H^2
\qquad
\forall\,\mathbf{u}\in\mathcal{V}.
\]
Recall that for $\mathbf{v}\in\mathcal{S}\setminus\{\mathbf{0}\}$ the lower
one-sided level and the lower principal level are
\[
\underline\lambda(\mathbf{v})
:= \sup_{\mathbf{u}\in\mathcal{S}^o}\mathcal{R}(\mathbf{u},\mathbf{v}),
\qquad
\underline\lambda
:= \inf_{\mathbf{v}\in\mathcal{S}\setminus\{\mathbf{0}\}}\underline\lambda(\mathbf{v}).
\]

\begin{lemma}[Compactness of normalized left sublevels]
	\label{lem:left-sublevel-compactness}
	Assume $\mathrm{(P)}$ and \eqref{eq:unifellipt}. Let
	$(\mathbf{v}_j)\subset\mathcal{S}$ and $(\mu_j)\subset\mathbb{R}$ satisfy
	\[
	\|\mathbf{v}_j\|_H=1,
	\qquad \sup_j\mu_j<+\infty,
	\qquad
	\langle\mathcal{L}\mathbf{v}_j,\mathbf{v}_j\rangle
	\le \mu_j(\mathcal{G}\mathbf{v}_j,\mathbf{v}_j)_H.
	\]
	Then $(\mathbf{v}_j)$ is bounded in $\mathcal{V}$, and up to a subsequence
	\[
	\mathbf{v}_j\rightharpoonup\mathbf{v}
	\quad\text{weakly in }\mathcal{V},
	\qquad
	\mathbf{v}_j\to\mathbf{v}
	\quad\text{strongly in }H,
	\]
	for some $\mathbf{v}\in\mathcal{S}$ with $\|\mathbf{v}\|_H=1$.
\end{lemma}

\begin{proof}
	Let $C_G>0$ be such that
	\[
	|(\mathcal G\mathbf u,\mathbf w)_H|
	\le C_G\|\mathbf u\|_H\|\mathbf w\|_H
	\qquad \forall\,\mathbf u,\mathbf w\in H.
	\]
	By the density of $\mathcal{S}^o$ in $\mathcal{S}\setminus\{\mathbf{0}\}$,
	for each $j$ we may choose $\mathbf w_{j,n}\in\mathcal S^o$ such that
	$\mathbf w_{j,n}\to\mathbf v_j$ in $\mathcal V$. Since
	$\mathbf v_j\in\mathcal S\setminus\{\mathbf0\}$ and $\mathrm{(P)}$ gives
	$(\mathcal G\mathbf v_j,\mathbf w_{j,n})_H>0$, passage to the limit yields
	$(\mathcal G\mathbf v_j,\mathbf v_j)_H\ge0$. Together with the sublevel
	inequality and $\|\mathbf{v}_j\|_H=1$, this gives
	\[
	\langle\mathcal{L}\mathbf{v}_j,\mathbf{v}_j\rangle
	\le \max\bigl\{\sup_j\mu_j,0\bigr\}\,C_G,
	\]
	so the G{\aa}rding inequality bounds $(\mathbf{v}_j)$ in $\mathcal{V}$.
	The remaining assertions follow from the reflexivity of $\mathcal{V}$, the
	compact embedding $\mathcal{V}\hookrightarrow H$, and the weak closedness of
	$\mathcal{S}$.
\end{proof}

\begin{theorem}[Attainment of the lower principal level]
	\label{thm:general-left-attainment}
	Assume $\mathrm{(P)}$ and \eqref{eq:unifellipt}. If $\underline\lambda<+\infty$,
	then $\underline\lambda\in\mathbb{R}$ and there exists
	$\mathbf{v}^*\in\mathcal{S}\setminus\{\mathbf{0}\}$ with $\|\mathbf{v}^*\|_H=1$
	such that $\underline\lambda(\mathbf{v}^*)=\underline\lambda$ and
	\begin{equation}\label{eq:continuous-left-quasi}
		\langle\mathcal{L}\mathbf{u},\mathbf{v}^*\rangle
		\le \underline\lambda\,(\mathcal{G}\mathbf{u},\mathbf{v}^*)_H
		\qquad \forall\,\mathbf{u}\in\mathcal{S}.
	\end{equation}
\end{theorem}

\begin{proof}
	By homogeneity, there exists a minimizing sequence
	$(\mathbf{v}_k)\subset\mathcal{S}$ with $\|\mathbf{v}_k\|_H=1$ and finite
	values $\lambda_k:=\underline\lambda(\mathbf{v}_k)$ such that
	$\lambda_k\to\underline\lambda$. The sequence $(\lambda_k)$ is bounded above.
	For each fixed $k$, the definition of $\lambda_k$ gives
	\[
	\langle\mathcal{L}\mathbf{u},\mathbf{v}_k\rangle
	\le \lambda_k(\mathcal{G}\mathbf{u},\mathbf{v}_k)_H
	\qquad \forall\,\mathbf{u}\in\mathcal{S}^o.
	\]
	By density and continuity, this inequality extends to all
	$\mathbf{u}\in\mathcal{S}$. Taking $\mathbf{u}=\mathbf{v}_k$ and applying
	Lemma~\ref{lem:left-sublevel-compactness}, we extract a subsequence with
	$\mathbf{v}_k\rightharpoonup\mathbf{v}^*$ weakly in $\mathcal{V}$ and
	$\mathbf{v}_k\to\mathbf{v}^*$ strongly in $H$, where
	$\mathbf{v}^*\in\mathcal{S}$ and $\|\mathbf{v}^*\|_H=1$.
	
	If $\underline\lambda=-\infty$, then for any fixed $\mathbf{u}_0\in\mathcal{S}^o$,
	\[
	\lambda_k \ge \mathcal{R}(\mathbf{u}_0,\mathbf{v}_k)
	\longrightarrow \mathcal{R}(\mathbf{u}_0,\mathbf{v}^*)\in\mathbb{R},
	\]
	a contradiction. Hence $\underline\lambda\in\mathbb{R}$.
	
	Passing to the limit in the sublevel inequality gives
	\[
	\langle\mathcal{L}\mathbf{u},\mathbf{v}^*\rangle
	\le \underline\lambda(\mathcal{G}\mathbf{u},\mathbf{v}^*)_H
	\qquad \forall\,\mathbf{u}\in\mathcal{S}^o.
	\]
	Dividing by the positive denominator gives
	$\underline\lambda(\mathbf{v}^*)\le\underline\lambda$. The reverse inequality
	holds by definition, so $\underline\lambda(\mathbf{v}^*)=\underline\lambda$.
	Density and continuity then extend the inequality to all
	$\mathbf{u}\in\mathcal{S}$, proving \eqref{eq:continuous-left-quasi}.
\end{proof}

\section{Continuous minimax limit}
\label{sec:pde}

Under the discrete positivity and boundary-exclusion assumptions, the
finite-dimensional theory yields exact positive right-left eigenpairs and
exact minimax identities on each discrete cone. We now pass to the Galerkin
limit. The proof shows that every convergent subsequence of discrete minimax
values produces a nontrivial right-left eigenpair of the continuous pencil;
strong positivity is then used to identify this limit with the continuous cone
minimax value.

Throughout this section we assume that \(\mathcal V\) is reflexive, the
embedding \(\mathcal V\hookrightarrow H\) is compact, and \(\mathcal S\) is
weakly closed in \(\mathcal V\). Introduce the dual cone
\[
\mathcal S^*
:=
\{\mathbf F\in\mathcal V^*:
\langle\mathbf F,\mathbf w\rangle\ge0
\ \forall\,\mathbf w\in\mathcal S\},
\]
and assume:

\begin{description}
	\item[$\mathrm{(SP)}$] \emph{Strong positivity of \(\mathcal L\).}
	If \(\mathbf w\in\mathcal S\setminus\{\mathbf0\}\) and
	\(\mathcal L\mathbf w\in\mathcal S^*\), then
	\(\mathbf w\in\mathcal S^o\).
	
	\item[$\mathrm{(SP^*)}$] \emph{Strong positivity of \(\mathcal L^*\).}
	If \(\mathbf w\in\mathcal S\setminus\{\mathbf0\}\) and
	\(\mathcal L^*\mathbf w\in\mathcal S^*\), then
	\(\mathbf w\in\mathcal S^o\).
\end{description}

\begin{remark}[Relation with the maximum principle]
	Conditions \(\mathrm{(SP)}\) and \(\mathrm{(SP^*)}\) are cone formulations
	of the strong positivity consequence of the maximum principle. For the
	standard positive cone, the inclusion
	\(\mathcal L\mathbf w\in\mathcal S^*\) means that
	\(\mathcal L\mathbf w\ge0\) in the weak sense. Thus \(\mathrm{(SP)}\) says
	that every nontrivial nonnegative supersolution is strictly positive:
	\[
	\mathbf w\in\mathcal S\setminus\{\mathbf0\},
	\qquad
	\mathcal L\mathbf w\ge0
	\quad\Longrightarrow\quad
	\mathbf w\in\mathcal S^o .
	\]
	In the scalar elliptic case this is precisely the usual strong maximum
	principle; see, for example, \cite{ProtterWeinberger}. For systems it is the
	corresponding strong positivity property with respect to the chosen cone,
	and it typically reflects an irreducible coupling of the components. The
	adjoint condition \(\mathrm{(SP^*)}\) is the same requirement for the left
	problem and is needed to obtain strict positivity of the limiting left
	eigenvector.
	
	The present formulation is also consistent with the classical
	maximum-principle approach to principal eigenvalues. In particular,
	Donsker--Varadhan \cite{Donsker} obtained a variational formula for the
	principal eigenvalue of a nonsymmetric operator satisfying a maximum
	principle, and the relation between principal eigenvalues and maximum
	principles for elliptic operators was further developed by
	Berestycki--Nirenberg--Varadhan \cite{BerestNienbVardan}. In the present
	argument the maximum-principle input has a more specific role: it is not used
	to define the principal value or to reduce the pencil to a positive
	resolvent. Rather, after the Galerkin minimax limit has produced nonnegative
	right-left quasi-eigenvectors, \(\mathrm{(SP)}\) and \(\mathrm{(SP^*)}\) are
	used to upgrade them to strictly positive right-left eigenvectors.
\end{remark}

\begin{theorem}[Convergence of the finite-dimensional minimax identities]
	\label{thm:continuous-minimax-limit}
	Assume \(\mathrm{(P)}\), \eqref{eq:unifellipt}, and the standing compactness
	hypotheses of this section. Suppose that, for every \(r\), the discrete
	minimax collapse lemma applies, and assume that, for some \(C>0\),
	\begin{equation}\label{eq:estM}
		0<\hat\lambda_r\le C
		\qquad
		\forall\,r\ge1.
	\end{equation}
	If, in addition, \(\mathrm{(SP)}\) and \(\mathrm{(SP^*)}\) hold, then the pencil
	\((\mathcal L,\mathcal G)\) satisfies the minimax condition
	\begin{equation}\label{eq:mimaPrin1}
		\hat\lambda
		:=
		\sup_{\mathbf u\in\mathcal S\setminus\{\mathbf0\}}
		\inf_{\mathbf v\in\mathcal S^o}\mathcal R(\mathbf u,\mathbf v)
		=
		\inf_{\mathbf v\in\mathcal S\setminus\{\mathbf0\}}
		\sup_{\mathbf u\in\mathcal S^o}\mathcal R(\mathbf u,\mathbf v)
		\ge
		0,
	\end{equation}
	and \(\hat\lambda_r\to\hat\lambda\). Let
	\(\mathbf u_r,\mathbf v_r\in\mathcal S_r^o\) be the corresponding discrete
	right-left eigenpairs, normalized by
	\[
	\|\mathbf u_r\|_{\mathcal V}
	=
	\|\mathbf v_r\|_{\mathcal V}
	=
	1.
	\]
	Then these normalized pairs have weak-in-\(\mathcal V\) and strong-in-\(H\)
	cluster points, and every such cluster point \((\mathbf u_*,\mathbf v_*)\)
	belongs to \(\mathcal S^o\times\mathcal S^o\), solves
	\[
	(\mathcal L-\hat\lambda\mathcal G)\mathbf u_*=0,
	\qquad
	(\mathcal L-\hat\lambda\mathcal G)^*\mathbf v_*=0,
	\]
	and satisfies \(\mathcal R(\mathbf u_*,\mathbf v_*)=\hat\lambda\). Hence
	every such cluster point is a principal right-left quasi-pair realizing
	the continuous minimax identity.
\end{theorem}

\begin{proof}
	By the discrete minimax collapse lemma,
	\(\mathbf u_r,\mathbf v_r\in\mathcal S_r^o\) satisfy
	\begin{equation}\label{eq:discrete-minimax-eigenpairs}
		L_r\bar{\mathbf u}_r
		=
		\hat\lambda_r G_r\bar{\mathbf u}_r,
		\qquad
		L_r^T\bar{\mathbf v}_r
		=
		\hat\lambda_r G_r^T\bar{\mathbf v}_r.
	\end{equation}
	We first show that the \(H\)-limits of normalized subsequences cannot vanish.
	If \(\|\mathbf u_r\|_H\to0\) along a subsequence, then testing the first
	equation in \eqref{eq:discrete-minimax-eigenpairs} by \(\mathbf u_r\) gives
	\[
	\langle\mathcal L\mathbf u_r,\mathbf u_r\rangle
	=
	\hat\lambda_r(\mathcal G\mathbf u_r,\mathbf u_r)_H
	\to
	0,
	\]
	which contradicts the G{\aa}rding inequality and the normalization
	\(\|\mathbf u_r\|_{\mathcal V}=1\). The same argument, using
	\[
	\bar{\mathbf v}_r^T L_r^T\bar{\mathbf v}_r
	=
	\bar{\mathbf v}_r^T L_r\bar{\mathbf v}_r,
	\]
	excludes \(\|\mathbf v_r\|_H\to0\).
	
	Let \(\lambda_*\) be any accumulation point of \((\hat\lambda_r)\). Passing
	to a subsequence, we may assume that \(\hat\lambda_r\to\lambda_*\). By
	\eqref{eq:estM}, \(\lambda_*\ge0\). By reflexivity, the compact embedding
	\(\mathcal V\hookrightarrow H\), and the normalization, we may further
	assume that
	\[
	\mathbf u_r\rightharpoonup\mathbf u_*,
	\qquad
	\mathbf v_r\rightharpoonup\mathbf v_*
	\quad\text{weakly in }\mathcal V,
	\]
	and
	\[
	\mathbf u_r\to\mathbf u_*,
	\qquad
	\mathbf v_r\to\mathbf v_*
	\quad\text{strongly in }H.
	\]
	The weak closedness of \(\mathcal S\), together with the nonvanishing of the
	\(H\)-limits, gives
	\[
	\mathbf u_*,\mathbf v_*
	\in
	\mathcal S\setminus\{\mathbf0\}.
	\]
	
	Let \(\mathbf z\in\mathcal V\). Choose
	\(\mathbf z_r\in\mathcal V_r\) as in \eqref{eq:Wr-density}. Passing to the
	limit in
	\[
	\langle\mathcal L\mathbf u_r,\mathbf z_r\rangle
	=
	\hat\lambda_r(\mathcal G\mathbf u_r,\mathbf z_r)_H
	\]
	gives
	\[
	\langle\mathcal L\mathbf u_*,\mathbf z\rangle
	=
	\lambda_*(\mathcal G\mathbf u_*,\mathbf z)_H.
	\]
	Similarly, the second equation in \eqref{eq:discrete-minimax-eigenpairs}
	gives
	\[
	\langle\mathcal L\mathbf z_r,\mathbf v_r\rangle
	=
	\hat\lambda_r(\mathcal G\mathbf z_r,\mathbf v_r)_H,
	\]
	and passage to the limit yields
	\[
	\langle\mathcal L\mathbf z,\mathbf v_*\rangle
	=
	\lambda_*(\mathcal G\mathbf z,\mathbf v_*)_H.
	\]
	Hence
	\begin{equation}\label{eq:LimEq}
		(\mathcal L-\lambda_*\mathcal G)\mathbf u_*=0,
		\qquad
		(\mathcal L-\lambda_*\mathcal G)^*\mathbf v_*=0.
	\end{equation}
	
	Condition \(\mathrm{(P)}\) and the density of
	\(\mathcal S^o\) in \(\mathcal S\setminus\{\mathbf0\}\) imply
	\[
	(\mathcal G\mathbf u_*,\mathbf w)_H\ge0,
	\qquad
	(\mathcal G\mathbf w,\mathbf v_*)_H\ge0
	\qquad
	\forall\,\mathbf w\in\mathcal S.
	\]
	Equivalently,
	\[
	\mathcal G\mathbf u_*,
	\mathcal G^*\mathbf v_*
	\in
	\mathcal S^*.
	\]
	Since \(\lambda_*\ge0\), \eqref{eq:LimEq} gives
	\[
	\mathcal L\mathbf u_*,
	\mathcal L^*\mathbf v_*
	\in
	\mathcal S^*.
	\]
	Thus \(\mathrm{(SP)}\) and \(\mathrm{(SP^*)}\) yield
	\[
	\mathbf u_*,\mathbf v_*\in\mathcal S^o.
	\]
	
	The two eigenvalue equations imply
	\[
	\mathcal R(\mathbf u_*,\mathbf v)=\lambda_*
	\quad
	\forall\,\mathbf v\in\mathcal S^o,
	\qquad
	\mathcal R(\mathbf u,\mathbf v_*)=\lambda_*
	\quad
	\forall\,\mathbf u\in\mathcal S^o.
	\]
	Therefore
	\begin{equation}\label{eq:limit-quasi-pair}
		\overline\lambda(\mathbf u_*)
		=
		\lambda_*
		=
		\underline\lambda(\mathbf v_*),
		\qquad
		\underline\lambda\le\lambda_*\le\overline\lambda,
	\end{equation}
	by Lemma~\ref{lem:main}. Since
	\(\mathbf u_*,\mathbf v_*\in\mathcal S^o\),
	Corollary~\ref{cor:interior-eigenvalue-collapse} gives
	\[
	\overline\lambda
	=
	\lambda_*
	=
	\underline\lambda.
	\]
	Every accumulation point of \((\hat\lambda_r)\) therefore equals the same
	value \(\hat\lambda\). Since \((\hat\lambda_r)\) is bounded by
	\eqref{eq:estM}, this proves
	\[
	\hat\lambda_r\to\hat\lambda.
	\]
	
	Finally, let \((\mathbf u_*,\mathbf v_*)\) be any weak-in-\(\mathcal V\) and
	strong-in-\(H\) cluster point of the normalized discrete eigenpairs. Since
	\(\hat\lambda_r\to\hat\lambda\), the preceding passage to the limit applies
	with \(\lambda_*=\hat\lambda\). Hence
	\[
	\mathbf u_*,\mathbf v_*
	\in
	\mathcal S^o,
	\]
	they solve
	\[
	(\mathcal L-\hat\lambda\mathcal G)\mathbf u_*=0,
	\qquad
	(\mathcal L-\hat\lambda\mathcal G)^*\mathbf v_*=0,
	\]
	and satisfy
	\[
	\mathcal R(\mathbf u_*,\mathbf v_*)=\hat\lambda.
	\]
	This proves the theorem.
\end{proof}

\begin{example}[A non-invertible pencil with a direct cone minimax level]
	The following example shows that the ordered-pencil quotient produces a
	principal cone level even when \(\mathcal L\) is not invertible and
	\(\mathcal G\) is genuinely coupled. It also clarifies the relation with
	shifted Krein--Rutman reductions: such reductions may be possible in
	special cases, but they require an auxiliary choice of a shift and a
	separate positivity analysis of the corresponding resolvent.
	
	Let \(\Omega=(0,\pi)\), \(\varphi_1(x)=\sin x\),
	\(\mathcal V=(H_0^1(0,\pi))^2\), \(H=(L^2(0,\pi))^2\), and
	\[
	\mathcal S
	=
	\{\mathbf u=(u_1,u_2)\in\mathcal V:\ u_1,u_2\ge0\},
	\qquad
	\mathcal S^o
	=
	\{\mathbf u\in\mathcal S:\ u_1,u_2>0\ \text{in }(0,\pi)\}.
	\]
	Set
	\[
	A=
	\begin{pmatrix}1&2\\2&1\end{pmatrix},
	\qquad
	G=
	\begin{pmatrix}3&1\\1&2\end{pmatrix},
	\qquad
	\mathcal L\mathbf u=-\mathbf u''+A\mathbf u,
	\qquad
	\mathcal G\mathbf u=G\mathbf u .
	\]
	Since all entries of \(G\) are positive,
	\[
	(\mathcal G\mathbf u,\mathbf v)_H>0
	\]
	for every admissible pair
	\[
	(\mathbf u,\mathbf v)\in
	((\mathcal S\setminus\{\mathbf0\})\times\mathcal S^o)
	\cup
	(\mathcal S^o\times(\mathcal S\setminus\{\mathbf0\})).
	\]
	
	The operator \(\mathcal L\) is not invertible. Indeed,
	\[
	A\binom{1}{-1}=-\binom{1}{-1},
	\]
	and, since \(-\varphi_1''=\varphi_1\),
	\[
	\mathbf z=\varphi_1\binom{1}{-1}
	\]
	satisfies \(\mathcal L\mathbf z=\mathbf0\). Hence the unshifted
	resolvent reduction \(\mathcal L^{-1}\mathcal G\) is not available.
	
	On the other hand, put
	\[
	p=\binom{1}{2},
	\qquad
	\mathbf u_*=\mathbf v_*=\varphi_1 p .
	\]
	Then
	\[
	(I+A)p=\binom{6}{6},
	\qquad
	Gp=\binom{5}{5},
	\]
	and therefore, with \(\lambda_*:=6/5\),
	\[
	\mathcal L\mathbf u_*=\lambda_*\mathcal G\mathbf u_* .
	\]
	Since \(A\) and \(G\) are symmetric, the same function is also a left
	eigenfunction:
	\[
	\mathcal L^*\mathbf v_*=\lambda_*\mathcal G^*\mathbf v_* .
	\]
	Moreover, \(\mathbf u_*,\mathbf v_*\in\mathcal S^o\).
	
	A shifted resolvent reduction is not automatic. Indeed, for \(\sigma=0\),
	\[
	(\mathcal L-\sigma\mathcal G)\varphi_1\binom{1}{-1}
	=
	\mathcal L\varphi_1\binom{1}{-1}
	=
	\mathbf0,
	\]
	whereas at the principal level \(\sigma=\lambda_*\),
	\[
	(\mathcal L-\sigma\mathcal G)\mathbf u_*
	=
	\mathcal L\mathbf u_*-\lambda_*\mathcal G\mathbf u_*
	=
	\mathbf0 .
	\]
	Thus the pencil is singular both at \(\sigma=0\) and at the value
	selected by the cone minimax formula. Consequently, a Krein--Rutman
	argument based on a shifted resolvent would first require choosing a
	non-spectral shift \(\sigma\) for which \(\mathcal L-\sigma\mathcal G\)
	is invertible, and then proving positivity, or strong positivity, of
	\[
	(\mathcal L-\sigma\mathcal G)^{-1}\mathcal G
	\]
	in the relevant cone. The cone minimax argument below avoids this
	auxiliary construction: it works directly with the pencil
	\((\mathcal L,\mathcal G)\) and its ordered quotient.
	
	For every \(\mathbf v\in\mathcal S^o\),
	\[
	\mathcal R(\mathbf u_*,\mathbf v)=\lambda_*,
	\]
	while for every \(\mathbf u\in\mathcal S\setminus\{\mathbf0\}\),
	\[
	\mathcal R(\mathbf u,\mathbf v_*)=\lambda_* .
	\]
	Hence, for every \(\mathbf u\in\mathcal S\setminus\{\mathbf0\}\),
	\[
	\inf_{\mathbf v\in\mathcal S^o}\mathcal R(\mathbf u,\mathbf v)
	\le
	\mathcal R(\mathbf u,\mathbf v_*)
	=
	\lambda_*,
	\]
	and equality is attained at \(\mathbf u=\mathbf u_*\). Therefore
	\[
	\sup_{\mathbf u\in\mathcal S\setminus\{\mathbf0\}}
	\inf_{\mathbf v\in\mathcal S^o}
	\mathcal R(\mathbf u,\mathbf v)
	=
	\lambda_* .
	\]
	Similarly, for every \(\mathbf v\in\mathcal S\setminus\{\mathbf0\}\),
	\[
	\sup_{\mathbf u\in\mathcal S^o}\mathcal R(\mathbf u,\mathbf v)
	\ge
	\mathcal R(\mathbf u_*,\mathbf v)
	=
	\lambda_*,
	\]
	and equality is attained at \(\mathbf v=\mathbf v_*\). Thus
	\[
	\hat\lambda
	=
	\sup_{\mathbf u\in\mathcal S\setminus\{\mathbf0\}}
	\inf_{\mathbf v\in\mathcal S^o}
	\mathcal R(\mathbf u,\mathbf v)
	=
	\inf_{\mathbf v\in\mathcal S\setminus\{\mathbf0\}}
	\sup_{\mathbf u\in\mathcal S^o}
	\mathcal R(\mathbf u,\mathbf v)
	=
	\frac65 .
	\]
	
	Thus the cone minimax method gives the principal level directly from the
	pencil \((\mathcal L,\mathcal G,\mathcal S)\), without first passing to a
	positive resolvent operator.
\end{example}

\begin{remark}[Relation to the Krein--Rutman principle]
	In the classical situation where the Krein--Rutman theorem applies to a
	suitable compact strongly positive resolvent associated with the pencil
	\[
	\mathcal L u=\lambda\mathcal G u,
	\]
	one obtains a positive right eigenvector, a positive adjoint eigenvector,
	and the corresponding principal spectral value. In this setting, the
	present minimax argument recovers the same value in variational form.
	Indeed, applying the interior collapse lemma to the positive right-left
	eigenpair gives the minimax identity
	\[
	\hat\lambda
	=
	\overline\lambda
	=
	\underline\lambda .
	\]
	Thus, under the hypotheses where the Krein--Rutman mechanism is available,
	the principal eigenvalue is also selected by the extended Rayleigh quotient
	and the corresponding cone minimax formula.
\end{remark}

\begin{remark}[Finite-dimensional minimax approximation]
	In the Krein--Rutman setting, the preceding result may be viewed as an
	alternative variational derivation of the principal-eigenvalue conclusion.
	A further point, however, is the approximation mechanism provided by the
	present approach. The principal right-left pair and the principal value are
	obtained as limits of finite-dimensional minimax saddle pairs:
	\[
	\hat\lambda_r
	\longrightarrow
	\hat\lambda
	=
	\overline\lambda
	=
	\underline\lambda .
	\]
	Thus the continuous minimax identity is recovered as the limit of exact
	finite-dimensional minimax identities. These finite-dimensional minimax
	problems can be studied and solved directly, without first reducing the
	pencil to a compact Krein--Rutman operator and without constructing
	Galerkin approximations of such a resolvent.
\end{remark}

\begin{remark}[Comparison with the maximum-principle approach]
	In the scalar setting of Berestycki--Nirenberg--Varadhan
	\cite{BerestNienbVardan}, the maximum principle yields the existence of a
	positive principal eigenfunction. In the overlap with that theory, this
	gives the positive right eigenfunction. The present theorem uses a different
	mechanism: it treats ordered pencils
	\(\mathcal L u=\lambda\mathcal G u\), possibly with singular \(\mathcal G\),
	and identifies a right-left principal pair through a two-variable cone
	minimax formula and its finite-dimensional minimax approximations.
\end{remark}

\section{A posteriori spectral enclosures and primal--dual gaps}
\label{sec:aposteriori-enclosures}

The minimax formulas turn cone trial vectors into certified lower and
upper bounds for the distinguished spectral levels. In finite dimensions
these bounds reduce to componentwise Collatz--Wielandt formulas requiring
only matrix--vector products, and their difference gives a computable
primal--dual gap and stopping criterion.

Recall that for $\mathbf{u},\mathbf{v}\in\mathcal{S}\setminus\{\mathbf{0}\}$,
\[
\overline\lambda(\mathbf{u})
:= \inf_{\mathbf{w}\in\mathcal{S}^o}\mathcal{R}(\mathbf{u},\mathbf{w}),
\qquad
\underline\lambda(\mathbf{v})
:= \sup_{\mathbf{w}\in\mathcal{S}^o}\mathcal{R}(\mathbf{w},\mathbf{v}).
\]
At the discrete level, for
$\mathbf{u},\mathbf{v}\in\mathcal{S}_r\setminus\{\mathbf{0}\}$, set
\begin{equation}\label{eq:discrete-trial-values}
	\lambda_{R,r}(\mathbf{u})
	:= \inf_{\mathbf{w}\in\mathcal{S}_r^o}\mathcal{R}_r(\mathbf{u},\mathbf{w}),
	\qquad
	\lambda_{L,r}(\mathbf{v})
	:= \sup_{\mathbf{w}\in\mathcal{S}_r^o}\mathcal{R}_r(\mathbf{w},\mathbf{v}).
\end{equation}

\subsection{Certified enclosures after minimax collapse}

\begin{corollary}[Continuous and discrete a posteriori enclosures]
	\label{cor:minimax-aposteriori-enclosures}
	Suppose the continuous minimax condition
	\(\overline\lambda=\underline\lambda=\hat\lambda\) holds. Then for every
	\(\mathbf{u},\mathbf{v}\in\mathcal{S}\setminus\{\mathbf{0}\}\),
	\begin{equation}\label{eq:continuous-certified-enclosure}
		\overline\lambda(\mathbf{u}) \le \hat\lambda \le \underline\lambda(\mathbf{v}).
	\end{equation}
	Setting
	\[
	\operatorname{gap}(\mathbf{u},\mathbf{v})
	:=
	\underline\lambda(\mathbf{v})-\overline\lambda(\mathbf{u})\ge0,
	\]
	one has
	\begin{equation}\label{eq:continuous-aposteriori-errors}
		0 \le \hat\lambda-\overline\lambda(\mathbf{u}) \le \operatorname{gap}(\mathbf{u},\mathbf{v}),
		\qquad
		0 \le \underline\lambda(\mathbf{v})-\hat\lambda \le \operatorname{gap}(\mathbf{u},\mathbf{v}).
	\end{equation}
	
	If the \(r\)-th discrete minimax collapse holds with value \(\hat\lambda_r\),
	then for every
	\(\mathbf{u},\mathbf{v}\in\mathcal{S}_r\setminus\{\mathbf{0}\}\),
	\begin{equation}\label{eq:discrete-certified-enclosure}
		\lambda_{R,r}(\mathbf{u}) \le \hat\lambda_r \le \lambda_{L,r}(\mathbf{v}),
	\end{equation}
	and, with
	\[
	\operatorname{gap}_r(\mathbf{u},\mathbf{v})
	:=
	\lambda_{L,r}(\mathbf{v})-\lambda_{R,r}(\mathbf{u}),
	\]
	\begin{equation}\label{eq:discrete-aposteriori-errors}
		0 \le \hat\lambda_r-\lambda_{R,r}(\mathbf{u}) \le \operatorname{gap}_r(\mathbf{u},\mathbf{v}),
		\qquad
		0 \le \lambda_{L,r}(\mathbf{v})-\hat\lambda_r \le \operatorname{gap}_r(\mathbf{u},\mathbf{v}).
	\end{equation}
	
	Under the hypotheses of Theorem~\ref{thm:continuous-minimax-limit}, let
	\(\widetilde{\mathbf{u}}_r,\widetilde{\mathbf{v}}_r
	\in\mathcal{S}_r\setminus\{\mathbf{0}\}\) be any trial elements such that
	\[
	\operatorname{gap}_r(\widetilde{\mathbf{u}}_r,\widetilde{\mathbf{v}}_r)\to0.
	\]
	Then
	\begin{equation}\label{eq:certificate-convergence}
		\lambda_{R,r}(\widetilde{\mathbf{u}}_r)\to\hat\lambda,
		\qquad
		\lambda_{L,r}(\widetilde{\mathbf{v}}_r)\to\hat\lambda.
	\end{equation}
\end{corollary}

\begin{proof}
	The continuous enclosure \eqref{eq:continuous-certified-enclosure} follows
	directly from
	\[
	\hat\lambda
	=
	\sup_{\mathbf{u}\in\mathcal S\setminus\{\mathbf0\}}
	\overline\lambda(\mathbf{u})
	=
	\inf_{\mathbf{v}\in\mathcal S\setminus\{\mathbf0\}}
	\underline\lambda(\mathbf{v}),
	\]
	and the discrete enclosure follows from the corresponding discrete minimax
	identity. The error estimates are immediate consequences of the two
	enclosures. Finally, Theorem~\ref{thm:continuous-minimax-limit} gives
	\(\hat\lambda_r\to\hat\lambda\), and combining this convergence with
	\eqref{eq:discrete-aposteriori-errors} gives
	\eqref{eq:certificate-convergence}.
\end{proof}

\subsection{Mixed discrete intervals and componentwise certificates}
\label{subsec:mixed-minimax-enclosures}

Before minimax collapse, the two mixed discrete minimax values
\(\underline\Lambda_r\) and \(\overline\Lambda_r\) still provide a certified
interval.

\begin{proposition}[Mixed enclosure and gap decomposition]
	\label{prop:mixed-minimax-enclosure}
	Assume $(P_r)$. Then $\underline\Lambda_r\le\overline\Lambda_r$, and every
	$\mathbf{u},\mathbf{v}\in\mathcal{S}_r^o$ satisfies
	\begin{equation}\label{eq:mixed-minimax-enclosure}
		\lambda_{R,r}(\mathbf{u})
		\le \underline\Lambda_r \le \overline\Lambda_r
		\le \lambda_{L,r}(\mathbf{v}).
	\end{equation}
	The total gap decomposes as
	\begin{equation}\label{eq:mixed-gap-decomposition}
		\operatorname{gap}_r(\mathbf{u},\mathbf{v})
		= \bigl[\lambda_{L,r}(\mathbf{v})-\overline\Lambda_r\bigr]
		+ \bigl[\overline\Lambda_r-\underline\Lambda_r\bigr]
		+ \bigl[\underline\Lambda_r-\lambda_{R,r}(\mathbf{u})\bigr],
	\end{equation}
	with all three terms nonnegative. In particular,
	\begin{equation}\label{eq:mixed-gap-control}
		0 \le \overline\Lambda_r-\underline\Lambda_r
		\le \operatorname{gap}_r(\mathbf{u},\mathbf{v}).
	\end{equation}
\end{proposition}

\begin{proof}
	For $\mathbf{u},\mathbf{v}\in\mathcal{S}_r^o$, density and continuity give
	\[
	\lambda_{R,r}(\mathbf{u})
	= \inf_{\mathbf{w}\in\mathcal{S}_r\setminus\{\mathbf{0}\}}
	\mathcal{R}_r(\mathbf{u},\mathbf{w}),
	\qquad
	\lambda_{L,r}(\mathbf{v})
	= \sup_{\mathbf{w}\in\mathcal{S}_r\setminus\{\mathbf{0}\}}
	\mathcal{R}_r(\mathbf{w},\mathbf{v}).
	\]
	The enclosure \eqref{eq:mixed-minimax-enclosure} follows from the definitions
	of $\underline\Lambda_r$ and $\overline\Lambda_r$, together with the minimax
	inequality $\underline\Lambda_r\le\overline\Lambda_r$. The decomposition
	\eqref{eq:mixed-gap-decomposition} is algebraic, and all three terms are
	nonnegative by \eqref{eq:mixed-minimax-enclosure}.
\end{proof}

The trial gap thus controls both the individual trial errors and the intrinsic
width of the mixed minimax interval. After collapse
$\underline\Lambda_r=\overline\Lambda_r$, the middle term vanishes and the gap
becomes a direct localization error for the collapsed mixed value.

\begin{proposition}[Componentwise Collatz--Wielandt enclosure]
	\label{prop:componentwise-mixed-enclosure}
	Assume $(P_r)$, and let
	$\mathbf{u},\mathbf{v}\in\mathcal{S}_r^o$ be any trial elements. Then
	$G_r\bar{\mathbf{u}}>0$ and $G_r^T\bar{\mathbf{v}}>0$ componentwise, and
	\begin{equation}\label{eq:componentwise-right-left-bound}
		\lambda_{R,r}(\mathbf{u})
		= \min_{1\le i\le mN_r}
		\frac{(L_r\bar{\mathbf{u}})_i}{(G_r\bar{\mathbf{u}})_i},
		\qquad
		\lambda_{L,r}(\mathbf{v})
		= \max_{1\le i\le mN_r}
		\frac{(L_r^T\bar{\mathbf{v}})_i}{(G_r^T\bar{\mathbf{v}})_i}.
	\end{equation}
	Consequently,
	\begin{equation}\label{eq:componentwise-mixed-enclosure}
		\min_i\frac{(L_r\bar{\mathbf{u}})_i}{(G_r\bar{\mathbf{u}})_i}
		\le \underline\Lambda_r \le \overline\Lambda_r
		\le \max_i\frac{(L_r^T\bar{\mathbf{v}})_i}{(G_r^T\bar{\mathbf{v}})_i}.
	\end{equation}
\end{proposition}

\begin{proof}
	Applying $(P_r)$ to coordinate directions gives
	$G_r\bar{\mathbf{u}}>0$ and $G_r^T\bar{\mathbf{v}}>0$ componentwise. Set
	$a_i:=(L_r\bar{\mathbf{u}})_i$ and $b_i:=(G_r\bar{\mathbf{u}})_i>0$. For
	$\bar{\mathbf{w}}>0$,
	\[
	\mathcal{R}_r(\mathbf{u},\mathbf{w})
	= \frac{\sum_i w_i b_i(a_i/b_i)}{\sum_i w_i b_i},
	\]
	a convex combination of the ratios $a_i/b_i$. Since the open cone can
	approach every coordinate direction, the infimum is $\min_i a_i/b_i$.
	The left formula is analogous, and \eqref{eq:componentwise-mixed-enclosure}
	follows from Proposition~\ref{prop:mixed-minimax-enclosure}.
\end{proof}

The bounds \eqref{eq:componentwise-right-left-bound} require only the four
products $L_r\bar{\mathbf{u}}$, $G_r\bar{\mathbf{u}}$, $L_r^T\bar{\mathbf{v}}$,
and $G_r^T\bar{\mathbf{v}}$, providing inexpensive certificates without first
solving the generalized matrix eigenvalue problem.

\section{Stability and approximation error estimates}
\label{sec:stability}

We record two consequences of the minimax characterization: one-sided
bounds under operator perturbations and estimates for the gap between the
continuous and finite-dimensional minimax levels.

\subsection{Operator perturbations}

Throughout this subsection \(\mathcal G\) is fixed. For a bounded linear
operator \(\mathcal L:\mathcal V\to\mathcal V^*\), set
\[
\mathcal R_{\mathcal L}(u,v)
:=
\frac{\langle \mathcal L u,v\rangle}{(\mathcal G u,v)_H},
\]
and define
\[
\overline\lambda_{\mathcal L}
:=
\sup_{u\in\mathcal S\setminus\{0\}}
\inf_{v\in\mathcal S^o}
\mathcal R_{\mathcal L}(u,v),
\qquad
\underline\lambda_{\mathcal L}
:=
\inf_{v\in\mathcal S\setminus\{0\}}
\sup_{u\in\mathcal S^o}
\mathcal R_{\mathcal L}(u,v).
\]
All minimax values below are assumed finite. For a bounded linear operator
\(\mathcal B:\mathcal V\to\mathcal V^*\) and
\(w\in\mathcal S\setminus\{0\}\), define
\begin{equation}\label{eq:dpm-G}
	d^+_{\mathcal B}(w)
	:=
	\inf_{v\in\mathcal S^o}
	\frac{\langle \mathcal B w,v\rangle}{(\mathcal G w,v)_H},
	\qquad
	d^-_{\mathcal B}(w)
	:=
	\sup_{u\in\mathcal S^o}
	\frac{\langle \mathcal B u,w\rangle}{(\mathcal G u,w)_H},
\end{equation}
where the denominators are positive by the mixed positivity assumption.

\begin{theorem}[One-sided perturbation estimates]
	\label{thm:stability-G}
	Let \(\mathcal L_1,\mathcal L_2:\mathcal V\to\mathcal V^*\) be bounded
	linear operators and set \(\mathcal B:=\mathcal L_1-\mathcal L_2\).
	\begin{enumerate}[label=\upshape(\roman*)]
		\item If \(u_{\mathcal L_2}\) is a right cone quasi-eigenvector realizing
		\(\overline\lambda_{\mathcal L_2}\) and
		\(d^+_{\mathcal B}(u_{\mathcal L_2})>-\infty\), then
		\begin{equation}\label{eq:upper_stability-G}
			\overline\lambda_{\mathcal L_1}
			-
			\overline\lambda_{\mathcal L_2}
			\ge
			d^+_{\mathcal B}(u_{\mathcal L_2}).
		\end{equation}
		
		\item If \(v_{\mathcal L_2}\) is a left cone quasi-eigenvector realizing
		\(\underline\lambda_{\mathcal L_2}\) and
		\(d^-_{\mathcal B}(v_{\mathcal L_2})<+\infty\), then
		\begin{equation}\label{eq:lower_stability-G}
			\underline\lambda_{\mathcal L_1}
			-
			\underline\lambda_{\mathcal L_2}
			\le
			d^-_{\mathcal B}(v_{\mathcal L_2}).
		\end{equation}
	\end{enumerate}
\end{theorem}

\begin{proof}
	For \((i)\), since \(u_{\mathcal L_2}\in\mathcal S\setminus\{0\}\),
	\[
	\begin{aligned}
		\overline\lambda_{\mathcal L_1}
		&\ge
		\inf_{v\in\mathcal S^o}
		\mathcal R_{\mathcal L_1}(u_{\mathcal L_2},v)
		\\
		&=
		\inf_{v\in\mathcal S^o}
		\left[
		\mathcal R_{\mathcal L_2}(u_{\mathcal L_2},v)
		+
		\frac{\langle \mathcal B u_{\mathcal L_2},v\rangle}
		{(\mathcal G u_{\mathcal L_2},v)_H}
		\right]
		\\
		&\ge
		\overline\lambda_{\mathcal L_2}
		+
		d^+_{\mathcal B}(u_{\mathcal L_2}).
	\end{aligned}
	\]
	For \((ii)\), using \(v_{\mathcal L_2}\in\mathcal S\setminus\{0\}\) and
	\(\sup(f+g)\le \sup f+\sup g\), we obtain
	\[
	\begin{aligned}
		\underline\lambda_{\mathcal L_1}
		&\le
		\sup_{u\in\mathcal S^o}
		\mathcal R_{\mathcal L_1}(u,v_{\mathcal L_2})
		\\
		&=
		\sup_{u\in\mathcal S^o}
		\left[
		\mathcal R_{\mathcal L_2}(u,v_{\mathcal L_2})
		+
		\frac{\langle \mathcal B u,v_{\mathcal L_2}\rangle}
		{(\mathcal G u,v_{\mathcal L_2})_H}
		\right]
		\\
		&\le
		\underline\lambda_{\mathcal L_2}
		+
		d^-_{\mathcal B}(v_{\mathcal L_2}).
	\end{aligned}
	\]
\end{proof}

\begin{corollary}[Perturbation band under the minimax condition]
	\label{cor:band-G}
	Suppose that both pencils \((\mathcal L_i,\mathcal G)\), \(i=1,2\), satisfy
	the minimax condition
	\[
	\overline\lambda_{\mathcal L_i}
	=
	\underline\lambda_{\mathcal L_i}
	=:
	\hat\lambda_{\mathcal L_i},
	\]
	and that \((\mathcal L_2,\mathcal G)\) admits a principal right-left
	quasi-pair \((u_{\mathcal L_2},v_{\mathcal L_2})\). Then, whenever the
	quantities below are finite,
	\begin{equation}\label{eq:band-final-G}
		\hat\lambda_{\mathcal L_2}
		+
		d^+_{\mathcal L_1-\mathcal L_2}(u_{\mathcal L_2})
		\le
		\hat\lambda_{\mathcal L_1}
		\le
		\hat\lambda_{\mathcal L_2}
		+
		d^-_{\mathcal L_1-\mathcal L_2}(v_{\mathcal L_2}).
	\end{equation}
\end{corollary}

\begin{proof}
	This follows directly from Theorem~\ref{thm:stability-G} and the two
	minimax equalities.
\end{proof}

For the one-parameter family
\(\mathcal L_\varepsilon:=\mathcal L+\varepsilon\mathcal B\),
\(\varepsilon>0\), Corollary~\ref{cor:band-G} gives the following
one-sided rate estimate.

\begin{corollary}[Small perturbations]
	\label{cor:stability_eps}
	Assume that \((\mathcal L,\mathcal G)\) and
	\((\mathcal L_\varepsilon,\mathcal G)\) satisfy the minimax condition for all
	sufficiently small \(\varepsilon>0\), that \((\mathcal L,\mathcal G)\) admits a
	principal right-left quasi-pair \((u_{\mathcal L},v_{\mathcal L})\), and that
	\(d^+_{\mathcal B}(u_{\mathcal L})>-\infty\) and
	\(d^-_{\mathcal B}(v_{\mathcal L})<+\infty\). Then
	\begin{equation}\label{eq:eps_bounds}
		\varepsilon d^+_{\mathcal B}(u_{\mathcal L})
		\le
		\hat\lambda_{\mathcal L_\varepsilon}
		-
		\hat\lambda_{\mathcal L}
		\le
		\varepsilon d^-_{\mathcal B}(v_{\mathcal L}).
	\end{equation}
	In particular,
	\(\hat\lambda_{\mathcal L_\varepsilon}\to\hat\lambda_{\mathcal L}\) as
	\(\varepsilon\to0^+\), and
	\[
	d^+_{\mathcal B}(u_{\mathcal L})
	\le
	\liminf_{\varepsilon\to0^+}
	\frac{\hat\lambda_{\mathcal L_\varepsilon}-\hat\lambda_{\mathcal L}}
	{\varepsilon}
	\le
	\limsup_{\varepsilon\to0^+}
	\frac{\hat\lambda_{\mathcal L_\varepsilon}-\hat\lambda_{\mathcal L}}
	{\varepsilon}
	\le
	d^-_{\mathcal B}(v_{\mathcal L}).
	\]
	Thus the interval
	\[
	\bigl[
	d^+_{\mathcal B}(u_{\mathcal L}),
	d^-_{\mathcal B}(v_{\mathcal L})
	\bigr]
	\]
	bounds the one-sided directional variation of
	\(\mathcal L\mapsto\hat\lambda_{\mathcal L}\) in the direction
	\(\mathcal B\).
\end{corollary}

\begin{proof}
	Apply Corollary~\ref{cor:band-G} with
	\(\mathcal L_1=\mathcal L_\varepsilon\) and
	\(\mathcal L_2=\mathcal L\). Since
	\(d^\pm_{\varepsilon\mathcal B}=\varepsilon d^\pm_{\mathcal B}\) for
	\(\varepsilon>0\), \eqref{eq:eps_bounds} follows. Dividing by
	\(\varepsilon\) and passing to one-sided limits gives the remaining
	inequalities.
\end{proof}

\subsection{Finite-dimensional approximation error}

Restricting \(\mathcal R_{\mathcal L}\) to the finite-dimensional cone, define
\[
\overline\lambda_{\mathcal L,r}
:=
\sup_{u\in\mathcal S_r\setminus\{0\}}
\inf_{v\in\mathcal S_r^o}
\mathcal R_{\mathcal L}(u,v),
\qquad
\underline\lambda_{\mathcal L,r}
:=
\inf_{v\in\mathcal S_r\setminus\{0\}}
\sup_{u\in\mathcal S_r^o}
\mathcal R_{\mathcal L}(u,v).
\]
For \(u_r,v_r\in\mathcal S_r\setminus\{0\}\), define the interpolation
defects
\begin{align}
	E^+_{\mathcal L,r}(u_r)
	&:=
	\inf_{v\in\mathcal S^o}
	\bigl[
	\mathcal R_{\mathcal L}(u_r,v)
	-
	\mathcal R_{\mathcal L}(u_r,\mathcal I_r v)
	\bigr],
	\label{eq:E-plus-short}\\
	E^-_{\mathcal L,r}(v_r)
	&:=
	\sup_{u\in\mathcal S^o}
	\bigl[
	\mathcal R_{\mathcal L}(u,v_r)
	-
	\mathcal R_{\mathcal L}(\mathcal I_r u,v_r)
	\bigr].
	\label{eq:E-minus-short}
\end{align}

\begin{theorem}[One-sided Galerkin error estimates]
	\label{thm:joint-error}
	Assume that \(\mathcal S_r\subset\mathcal S\), that the denominators in the
	continuous and discrete quotients are positive on the corresponding
	admissible pairs, and that
	\(\mathcal I_r(\mathcal S^o)\subset\mathcal S_r^o\). Let
	\(u_r,v_r\in\mathcal S_r\setminus\{0\}\) attain the discrete extremal levels:
	\[
	\inf_{w\in\mathcal S_r^o}
	\mathcal R_{\mathcal L}(u_r,w)
	=
	\overline\lambda_{\mathcal L,r},
	\qquad
	\sup_{w\in\mathcal S_r^o}
	\mathcal R_{\mathcal L}(w,v_r)
	=
	\underline\lambda_{\mathcal L,r}.
	\]
	If
	\(E^+_{\mathcal L,r}(u_r)>-\infty\) and
	\(E^-_{\mathcal L,r}(v_r)<+\infty\), then
	\begin{equation}\label{eq:joint-error-estimates}
		\overline\lambda_{\mathcal L}
		-
		\overline\lambda_{\mathcal L,r}
		\ge
		E^+_{\mathcal L,r}(u_r),
		\qquad
		\underline\lambda_{\mathcal L}
		-
		\underline\lambda_{\mathcal L,r}
		\le
		E^-_{\mathcal L,r}(v_r).
	\end{equation}
\end{theorem}

\begin{proof}
	Since \(u_r\in\mathcal S_r\setminus\{0\}\subset\mathcal S\setminus\{0\}\),
	\[
	\begin{aligned}
		\overline\lambda_{\mathcal L}
		&\ge
		\inf_{v\in\mathcal S^o}
		\mathcal R_{\mathcal L}(u_r,v)
		\\
		&=
		\inf_{v\in\mathcal S^o}
		\bigl[
		\mathcal R_{\mathcal L}(u_r,\mathcal I_r v)
		+
		\mathcal R_{\mathcal L}(u_r,v)
		-
		\mathcal R_{\mathcal L}(u_r,\mathcal I_r v)
		\bigr]
		\\
		&\ge
		\inf_{w\in\mathcal S_r^o}
		\mathcal R_{\mathcal L}(u_r,w)
		+
		E^+_{\mathcal L,r}(u_r)
		\\
		&=
		\overline\lambda_{\mathcal L,r}
		+
		E^+_{\mathcal L,r}(u_r).
	\end{aligned}
	\]
	Similarly, since
	\(v_r\in\mathcal S_r\setminus\{0\}\subset\mathcal S\setminus\{0\}\) and
	\(\sup(f+g)\le \sup f+\sup g\),
	\[
	\begin{aligned}
		\underline\lambda_{\mathcal L}
		&\le
		\sup_{u\in\mathcal S^o}
		\mathcal R_{\mathcal L}(u,v_r)
		\\
		&=
		\sup_{u\in\mathcal S^o}
		\bigl[
		\mathcal R_{\mathcal L}(\mathcal I_r u,v_r)
		+
		\mathcal R_{\mathcal L}(u,v_r)
		-
		\mathcal R_{\mathcal L}(\mathcal I_r u,v_r)
		\bigr]
		\\
		&\le
		\sup_{w\in\mathcal S_r^o}
		\mathcal R_{\mathcal L}(w,v_r)
		+
		E^-_{\mathcal L,r}(v_r)
		\\
		&=
		\underline\lambda_{\mathcal L,r}
		+
		E^-_{\mathcal L,r}(v_r).
	\end{aligned}
	\]
\end{proof}

\begin{corollary}[Two-sided error under the minimax condition]
	\label{cor:two-sided-error-short}
	Assume the hypotheses of Theorem~\ref{thm:joint-error}. If the continuous
	and finite-dimensional levels collapse,
	\[
	\overline\lambda_{\mathcal L}
	=
	\underline\lambda_{\mathcal L}
	=:
	\hat\lambda_{\mathcal L},
	\qquad
	\overline\lambda_{\mathcal L,r}
	=
	\underline\lambda_{\mathcal L,r}
	=:
	\hat\lambda_{\mathcal L,r},
	\]
	then
	\begin{equation}\label{eq:two-sided-error-short}
		E^+_{\mathcal L,r}(u_r)
		\le
		\hat\lambda_{\mathcal L}
		-
		\hat\lambda_{\mathcal L,r}
		\le
		E^-_{\mathcal L,r}(v_r).
	\end{equation}
\end{corollary}

\begin{proof}
	This follows directly from Theorem~\ref{thm:joint-error}.
\end{proof}

\section{A genuinely coupled non-selfadjoint model with singular weight}
\label{sec:complete-singular-weight-model}

We present a model in which all assumptions of
Theorem~\ref{thm:continuous-minimax-limit} can be verified explicitly, while
neither the continuous nor the discrete minimax value is available in closed
form. Thus the convergence of the finite-dimensional minimax identities
follows from the general theorem rather than from an independently known
spectral formula. A constant symmetric-coupling subcase, treated at the end,
provides an exactly solvable benchmark and shows that the perturbation
estimate is sharp.

Let $\mathcal{V}=(H_0^1(0,1))^2$ and $H=[L^2(0,1)]^2$, equipped with the
componentwise positive cone
\begin{align*}
	&\mathcal S=\{\mathbf u=(u_1,u_2)\in\mathcal V: u_1,u_2\ge0\text{ a.e.}\},\\
	&\mathcal S^o=\{\mathbf u\in\mathcal V\cap[C([0,1])]^2: u_1,u_2>0\text{ in }(0,1)\}.
\end{align*}
Since $H_0^1(0,1)\hookrightarrow C([0,1])$, this is the standard cone of
nonnegative pairs together with its strictly positive part.

Fix $b\in\mathbb R\setminus\{0\}$ and consider the scalar Dirichlet operator
\begin{equation}\label{eq:complete-model-scalar-operator}
	\mathcal Au=-u''+bu', \qquad u(0)=u(1)=0,
\end{equation}
with adjoint \(\mathcal A^*v=-v''-bv'\). Its principal right-left eigenpair is
\begin{equation}\label{eq:complete-model-scalar-eigenpair}
	\mu_1=\pi^2+\tfrac{b^2}{4}, \qquad
	\phi_1(x)=e^{bx/2}\sin(\pi x), \qquad
	\psi_1(x)=e^{-bx/2}\sin(\pi x),
\end{equation}
so $\mathcal A\phi_1=\mu_1\phi_1$, $\mathcal A^*\psi_1=\mu_1\psi_1$, and
$\phi_1,\psi_1>0$ in $(0,1)$.

Let $\beta_1,\beta_2\in C([0,1])$ satisfy
\begin{equation}\label{eq:complete-model-coupling-bounds}
	0<\beta_-\le\beta_i(x)\le\beta_+<\mu_1,
	\qquad x\in[0,1],\quad i=1,2,
\end{equation}
with $\beta_1\not\equiv\beta_2$. Define
\begin{equation}\label{eq:complete-model-pencil}
	\mathcal L_{\boldsymbol\beta}
	=\begin{pmatrix}\mathcal A&-\beta_1(x)I\\-\beta_2(x)I&\mathcal A\end{pmatrix},
	\qquad
	\mathcal G=\begin{pmatrix}I&I\\I&I\end{pmatrix},
	\qquad
	\boldsymbol\beta=(\beta_1,\beta_2).
\end{equation}
The operator $\mathcal L_{\boldsymbol\beta}$ is non-selfadjoint because
$b\ne0$, and the unequal off-diagonal coefficients make the coupling
nonsymmetric. Moreover, the diagonal subspace
$\{\phi(1,1)^T:\phi\in H_0^1(0,1)\}$ is not invariant under
$\mathcal L_{\boldsymbol\beta}$, so the system does not reduce to the scalar
problem \eqref{eq:complete-model-scalar-operator}.

The weight $\mathcal G$ is singular, since
$\ker\mathcal G=\{(w,-w)^T:w\in L^2(0,1)\}$. Nevertheless,
\[
(\mathcal G\mathbf u,\mathbf v)_H
=\int_0^1(u_1+u_2)(v_1+v_2)\,dx>0
\]
for every admissible mixed cone pair; hence condition $(P)$ holds. The uniform
ellipticity condition \eqref{eq:unifellipt} is immediate because the principal
part of $\mathcal L_{\boldsymbol\beta}$ is $-d^2/dx^2$ times the identity
matrix, and the density property \eqref{eq:Wr-density} is standard for the
conforming finite element spaces introduced below.

\smallskip
\noindent\textit{Strong positivity.}
Conditions $(SP)$ and $(SP^*)$ hold. Suppose
$\mathbf w=(w_1,w_2)\in\mathcal S\setminus\{\mathbf0\}$ and
$\mathcal L_{\boldsymbol\beta}\mathbf w\in\mathcal S^*$. Then, in the weak
sense,
\[
\mathcal Aw_1\ge\beta_1w_2\ge0, \qquad
\mathcal Aw_2\ge\beta_2w_1\ge0.
\]
Neither component can vanish identically. Indeed, if $w_1\equiv0$, then the
first inequality gives $-\beta_1w_2\ge0$, and hence $w_2\equiv0$, which
contradicts $\mathbf w\ne\mathbf0$; the case $w_2\equiv0$ is the same.
The scalar strong maximum principle therefore yields $w_1,w_2>0$ in $(0,1)$.
Thus $\mathbf w\in\mathcal S^o$. Since
\[
\mathcal L_{\boldsymbol\beta}^*
=\begin{pmatrix}\mathcal A^*&-\beta_2(x)I\\-\beta_1(x)I&\mathcal A^*\end{pmatrix},
\]
the same argument proves $(SP^*)$.

\smallskip
\noindent\textit{Discrete setting.}
Let $0=x_0<x_1<\cdots<x_{N_r+1}=1$ be the uniform partition with mesh size
$h_r=1/(N_r+1)$, and let $W_r\subset H_0^1(0,1)$ be the conforming $P_1$
space with nodal basis $\{\psi_j\}_{j=1}^{N_r}$. Denote by $A_r$ and $M_r$
the scalar convection--diffusion and mass matrices:
\[
(A_r)_{ij}=\int_0^1\psi_j'\psi_i'\,dx+b\int_0^1\psi_j'\psi_i\,dx,
\qquad
(M_r)_{ij}=\int_0^1\psi_j\psi_i\,dx.
\]
On the uniform mesh,
\begin{equation}\label{eq:complete-model-scalar-matrix}
	(A_r)_{jj}=\tfrac{2}{h_r}, \qquad
	(A_r)_{j,j-1}=-\tfrac{1}{h_r}-\tfrac{b}{2}, \qquad
	(A_r)_{j,j+1}=-\tfrac{1}{h_r}+\tfrac{b}{2}.
\end{equation}
The strict mesh P\'eclet condition
\begin{equation}\label{eq:complete-model-peclet}
	|b|h_r<2
\end{equation}
makes $A_r$ an irreducible $Z$-matrix. Its symmetric part is positive
definite; hence $A_r$ is positive stable and therefore an irreducible
nonsingular $M$-matrix. It follows that $A_r^{-1}>0$ and $A_r^{-1}M_r>0$.
The Perron--Frobenius theorem then yields a principal generalized eigenvalue
$\mu_{1,r}>0$ and positive vectors $\bar\phi_{1,r},\bar\psi_{1,r}>0$ satisfying
\begin{equation}\label{eq:complete-model-discrete-scalar-eigenpairs}
	A_r\bar\phi_{1,r}=\mu_{1,r}M_r\bar\phi_{1,r}, \qquad
	A_r^T\bar\psi_{1,r}=\mu_{1,r}M_r\bar\psi_{1,r}.
\end{equation}
Standard conforming spectral approximation gives $\mu_{1,r}\to\mu_1$.
Discarding finitely many initial meshes and relabeling the remaining sequence,
we may assume
\begin{equation}\label{eq:complete-model-fine-mesh}
	|b|h_r<2 \qquad\text{and}\qquad \beta_+<\mu_{1,r}
	\qquad \forall\,r\ge1.
\end{equation}

Set
\[
(B_{i,r})_{jk}:=\int_0^1\beta_i(x)\psi_k(x)\psi_j(x)\,dx,
\qquad i=1,2.
\]
The two-component discrete pencil matrices are
\begin{equation}\label{eq:complete-model-discrete-pencil}
	L_{\boldsymbol\beta,r}
	=\begin{pmatrix}A_r&-B_{1,r}\\-B_{2,r}&A_r\end{pmatrix},
	\qquad
	G_r=\begin{pmatrix}M_r&M_r\\M_r&M_r\end{pmatrix}.
\end{equation}
For every discrete admissible mixed pair,
\[
\bar{\mathbf v}^{\,T}G_r\bar{\mathbf u}
=(\bar u_1+\bar u_2)^TM_r(\bar v_1+\bar v_2)>0,
\]
so $(P_r)$ holds. By \eqref{eq:complete-model-coupling-bounds}, for every
$\bar{\mathbf z}\ge0$,
\begin{equation}\label{eq:complete-model-weighted-mass-bounds}
	\beta_-M_r\bar{\mathbf z}\le B_{i,r}\bar{\mathbf z}\le\beta_+M_r\bar{\mathbf z},
	\qquad i=1,2,
\end{equation}
componentwise. Let $\mathbf q_r$ be the finite element vector with coefficient
vector
\[
\bar{\mathbf q}_r=(\bar\phi_{1,r},\bar\phi_{1,r})^T>0.
\]
Then
\[
L_{\boldsymbol\beta,r}\bar{\mathbf q}_r
=
\binom{\mu_{1,r}M_r\bar\phi_{1,r}-B_{1,r}\bar\phi_{1,r}}
{\mu_{1,r}M_r\bar\phi_{1,r}-B_{2,r}\bar\phi_{1,r}}
\ge
(\mu_{1,r}-\beta_+)
\binom{M_r\bar\phi_{1,r}}{M_r\bar\phi_{1,r}}>0.
\]
The matrix $L_{\boldsymbol\beta,r}$ is an irreducible $Z$-matrix: the diagonal
blocks are irreducible $Z$-matrices, and the off-diagonal blocks
$-B_{i,r}$, with $\beta_i\ge\beta_->0$, couple the two components. Since
there exists $\bar{\mathbf q}_r>0$ with
$L_{\boldsymbol\beta,r}\bar{\mathbf q}_r>0$, the matrix
$L_{\boldsymbol\beta,r}$ is an irreducible nonsingular $M$-matrix. The same is
true for $L_{\boldsymbol\beta,r}^T$. Consequently,
\[
\bar{\mathbf w}\ge0,\quad\bar{\mathbf w}\ne0,\quad
L_{\boldsymbol\beta,r}\bar{\mathbf w}\ge0
\;\Longrightarrow\;\bar{\mathbf w}>0,
\]
and the same implication holds with $L_{\boldsymbol\beta,r}^T$ in place of
$L_{\boldsymbol\beta,r}$. Thus condition $(B_r)$ holds.

\smallskip
\noindent\textit{Uniform bounds.}
By Proposition~\ref{prop:componentwise-mixed-enclosure} and
\eqref{eq:complete-model-weighted-mass-bounds},
\[
\lambda_{R,r}(\mathbf q_r)
=
\min_{1\le j\le2N_r}
\frac{(L_{\boldsymbol\beta,r}\bar{\mathbf q}_r)_j}
{(G_r\bar{\mathbf q}_r)_j}
\ge
\frac{\mu_{1,r}-\beta_+}{2}>0.
\]
Hence $\overline\lambda_r\ge\lambda_{R,r}(\mathbf q_r)>0$.
Lemma~\ref{lem:discrete-collapse} yields positive right-left eigenvectors
$\mathbf u_r,\mathbf v_r\in\mathcal S_r^o$ and the common discrete minimax
value
\begin{equation}\label{eq:complete-model-discrete-minimax-value}
	\hat\lambda_r
	=
	\sup_{\mathbf u\in\mathcal S_r\setminus\{\mathbf0\}}
	\inf_{\mathbf v\in\mathcal S_r^o}\mathcal R_r(\mathbf u,\mathbf v)
	=
	\inf_{\mathbf v\in\mathcal S_r\setminus\{\mathbf0\}}
	\sup_{\mathbf u\in\mathcal S_r^o}\mathcal R_r(\mathbf u,\mathbf v).
\end{equation}
For the upper bound, let $\mathbf p_r$ be the finite element vector with
coefficient vector
\[
\bar{\mathbf p}_r=(\bar\psi_{1,r},\bar\psi_{1,r})^T>0.
\]
Since $B_{i,r}^T=B_{i,r}$,
\[
\lambda_{L,r}(\mathbf p_r)
=
\max_{1\le j\le2N_r}
\frac{(L_{\boldsymbol\beta,r}^T\bar{\mathbf p}_r)_j}
{(G_r^T\bar{\mathbf p}_r)_j}
\le
\frac{\mu_{1,r}-\beta_-}{2}.
\]
Together these estimates give
\begin{equation}\label{eq:complete-model-discrete-bounds}
	0<\frac{\mu_{1,r}-\beta_+}{2}
	\le
	\hat\lambda_r
	\le
	\frac{\mu_{1,r}-\beta_-}{2}.
\end{equation}
Since $\mu_{1,r}\to\mu_1$ and $\beta_+<\mu_1$, there exists $\nu>0$ such that
\[
0<\nu^{-1}\le\hat\lambda_r\le\nu
\qquad \forall\,r\ge1.
\]

All hypotheses of Theorem~\ref{thm:continuous-minimax-limit} are satisfied.
Moreover, the positive lower bound in \eqref{eq:complete-model-discrete-bounds}
shows that the limiting value is strictly positive. Hence
\begin{equation}\label{eq:complete-model-minimax-convergence}
	\hat\lambda_r\longrightarrow\hat\lambda>0,
\end{equation}
where
\begin{equation}\label{eq:complete-model-continuous-minimax}
	\hat\lambda
	=
	\sup_{\mathbf u\in\mathcal S\setminus\{\mathbf0\}}
	\inf_{\mathbf v\in\mathcal S^o}\mathcal R_{\boldsymbol\beta}(\mathbf u,\mathbf v)
	=
	\inf_{\mathbf v\in\mathcal S\setminus\{\mathbf0\}}
	\sup_{\mathbf u\in\mathcal S^o}\mathcal R_{\boldsymbol\beta}(\mathbf u,\mathbf v),
\end{equation}
and every weak-in-\(\mathcal V\) and strong-in-\(H\) cluster point of the
normalized discrete right-left eigenpairs lies in
$\mathcal S^o\times\mathcal S^o$ and solves
\[
(\mathcal L_{\boldsymbol\beta}-\hat\lambda\mathcal G)\mathbf u_*=0,
\qquad
(\mathcal L_{\boldsymbol\beta}-\hat\lambda\mathcal G)^*\mathbf v_*=0.
\]

The conclusion concerns the minimax formulas themselves, not only the
eigenvalues of the Galerkin matrix pencils. After the discrete collapse,
Proposition~\ref{prop:componentwise-mixed-enclosure} gives
\begin{equation}\label{eq:complete-model-componentwise-minimax}
	\hat\lambda_r
	=
	\sup_{\mathbf u\in\mathcal S_r^o}
	\min_{1\le j\le2N_r}
	\frac{(L_{\boldsymbol\beta,r}\bar{\mathbf u})_j}
	{(G_r\bar{\mathbf u})_j}
	=
	\inf_{\mathbf v\in\mathcal S_r^o}
	\max_{1\le j\le2N_r}
	\frac{(L_{\boldsymbol\beta,r}^T\bar{\mathbf v})_j}
	{(G_r^T\bar{\mathbf v})_j}.
\end{equation}
Combining
\eqref{eq:complete-model-minimax-convergence}--\eqref{eq:complete-model-componentwise-minimax},
we obtain
\begin{align}
	\hat\lambda
	&=\lim_{r\to\infty}
	\sup_{\mathbf u\in\mathcal S_r^o}
	\min_{1\le j\le2N_r}
	\frac{(L_{\boldsymbol\beta,r}\bar{\mathbf u})_j}
	{(G_r\bar{\mathbf u})_j}
	\nonumber\\
	&=\lim_{r\to\infty}
	\inf_{\mathbf v\in\mathcal S_r^o}
	\max_{1\le j\le2N_r}
	\frac{(L_{\boldsymbol\beta,r}^T\bar{\mathbf v})_j}
	{(G_r^T\bar{\mathbf v})_j}
	\nonumber\\
	&=
	\sup_{\mathbf u\in\mathcal S\setminus\{\mathbf0\}}
	\inf_{\mathbf v\in\mathcal S^o}
	\mathcal R_{\boldsymbol\beta}(\mathbf u,\mathbf v)
	=
	\inf_{\mathbf v\in\mathcal S\setminus\{\mathbf0\}}
	\sup_{\mathbf u\in\mathcal S^o}
	\mathcal R_{\boldsymbol\beta}(\mathbf u,\mathbf v).
	\label{eq:complete-model-variational-convergence}
\end{align}
Thus the continuous minimax value is approximated by finite-dimensional
minimax optimization, without requiring an explicit formula for the
eigenpairs. In particular, every
$\mathbf u,\mathbf v\in\mathcal S_r^o$ gives the computable two-sided
enclosure
\begin{equation}\label{eq:complete-model-CW-enclosure}
	\min_{1\le j\le2N_r}
	\frac{(L_{\boldsymbol\beta,r}\bar{\mathbf u})_j}
	{(G_r\bar{\mathbf u})_j}
	\le
	\hat\lambda_r
	\le
	\max_{1\le j\le2N_r}
	\frac{(L_{\boldsymbol\beta,r}^T\bar{\mathbf v})_j}
	{(G_r^T\bar{\mathbf v})_j},
\end{equation}
and the gap between the two bounds serves as an a posteriori stopping
criterion for the finite-dimensional minimax problem.

\medskip
\noindent\textbf{Exactly solvable symmetric-coupling subcase.}
The preceding argument requires no explicit representation of the eigenpairs.
Such a representation reappears if the nonsymmetry condition is dropped and
one takes
\[
\beta_1(x)=\beta_2(x)=\beta,\qquad 0<\beta<\mu_1.
\]
Then the diagonal subspace is invariant and
\[
\mathcal L_\beta
=
\begin{pmatrix}\mathcal A&-\beta I\\-\beta I&\mathcal A\end{pmatrix}.
\]
With $e_+=(1,1)^T$, the eigenpairs are
$\mathbf u_*=\phi_1e_+$, $\mathbf v_*=\psi_1e_+$, and
\begin{equation}\label{eq:complete-model-explicit-continuous-value}
	\hat\lambda(\beta)=\frac{\mu_1-\beta}{2}.
\end{equation}
Likewise,
\[
\bar{\mathbf u}_r=(\bar\phi_{1,r},\bar\phi_{1,r})^T,
\qquad
\bar{\mathbf v}_r=(\bar\psi_{1,r},\bar\psi_{1,r})^T,
\]
and
\begin{equation}\label{eq:complete-model-explicit-discrete-value}
	\hat\lambda_r(\beta)=\frac{\mu_{1,r}-\beta}{2}
	\longrightarrow
	\frac{\mu_1-\beta}{2}
	=
	\hat\lambda(\beta),
\end{equation}
confirming the general convergence theorem in this exactly solvable case.

This subcase also shows that the perturbation estimate is sharp. Let
$\mathcal L_{\beta+\varepsilon}=\mathcal L_\beta+\varepsilon\mathcal P$ with
$\mathcal P\mathbf u=-(u_2,u_1)^T$, where $\varepsilon>0$ and
$\beta+\varepsilon<\mu_1$. Then
\[
\hat\lambda(\beta+\varepsilon)-\hat\lambda(\beta)=-\frac{\varepsilon}{2}.
\]
Since $\mathcal P\mathbf u_*=-\phi_1e_+$ and
$\mathcal G\mathbf u_*=2\phi_1e_+$,
\[
\frac{\langle\mathcal P\mathbf u_*,\mathbf v\rangle}
{(\mathcal G\mathbf u_*,\mathbf v)_H}
=
-\frac12
\qquad\forall\,\mathbf v\in\mathcal S^o,
\]
and similarly
\[
\frac{\langle\mathcal P\mathbf u,\mathbf v_*\rangle}
{(\mathcal G\mathbf u,\mathbf v_*)_H}
=
-\frac12
\qquad\forall\,\mathbf u\in\mathcal S^o.
\]
Hence
\[
d_{\mathcal P}^+(\mathbf u_*)=
d_{\mathcal P}^-(\mathbf v_*)=-\frac12,
\]
and
\begin{equation}\label{eq:complete-model-sharpness}
	\hat\lambda(\beta+\varepsilon)-\hat\lambda(\beta)
	=
	d_{\varepsilon\mathcal P}^+(\mathbf u_*)
	=
	d_{\varepsilon\mathcal P}^-(\mathbf v_*)
	=
	-\frac{\varepsilon}{2}.
\end{equation}
Both sides of the perturbation band are thus attained. For every sufficiently
fine mesh with $\beta+\varepsilon<\mu_{1,r}$,
\[
\hat\lambda_r(\beta+\varepsilon)-\hat\lambda_r(\beta)
=
-\frac{\varepsilon}{2}.
\]


\section{A non-cooperative skew coupling}
\label{sec:exampl}

We give a continuous example in which the standard positive cone is not
invariant, yet the cone minimax value is determined explicitly by the
principal eigenvalue of an underlying scalar operator. The example
illustrates the continuous minimax principle outside the cooperative
setting; discrete boundary-exclusion assumptions are not considered here.

Let
\[
\mathcal L\mathbf u
=
-\Delta\mathbf u+(b\cdot\nabla)\mathbf u
+(\mu I_2+K(x))\mathbf u,
\qquad
K(x)=\begin{pmatrix}0&\omega(x)\\-\omega(x)&0\end{pmatrix},
\]
in $\Omega$ with homogeneous Dirichlet boundary conditions, where
$b\in L^\infty(\Omega;\mathbb R^n)$, $\omega\in L^\infty(\Omega)$, and
$\omega\ge0$ a.e. Set
\[
\mathcal L_0\phi:=-\Delta\phi+b\cdot\nabla\phi+\mu\phi,
\]
and assume $\mathcal L_0\phi_1=\lambda_1\phi_1$,
$\mathcal L_0^*\psi_1=\lambda_1\psi_1$, with $\phi_1,\psi_1>0$ in $\Omega$.
We use the componentwise positive cone
\[
\mathcal S=\{\mathbf u=(u_1,u_2):u_1,u_2\ge0\},
\qquad
\mathcal S^o=\{\mathbf u\in\mathcal S:u_1,u_2>0\text{ in }\Omega\},
\]
and the Rayleigh quotient
\[
\mathcal R(\mathbf u,\mathbf v)
:=
\frac{\langle\mathcal L\mathbf u,\mathbf v\rangle}{(\mathbf u,\mathbf v)_H}.
\]

\begin{proposition}[Fixed-sign skew coupling]
	\label{prop:skew-drift-fixed-sign}
	If $\omega\ge0$ a.e.\ in $\Omega$, then
	$\overline\lambda=\underline\lambda=\lambda_1$. The boundary vectors
	\[
	\boldsymbol\phi^*:=(0,\phi_1)^T, \qquad
	\boldsymbol\psi^*:=(0,\psi_1)^T
	\]
	realize the right and left minimax levels, respectively, and
	$(\boldsymbol\phi^*,\boldsymbol\psi^*)_H=(\phi_1,\psi_1)_{L^2}>0$.
\end{proposition}

\begin{proof}
	For $\mathbf u=(u_1,u_2)^T$ and $\mathbf v=(v_1,v_2)^T$,
	\[
	\mathcal R(\mathbf u,\mathbf v)
	=
	\frac{
		\langle\mathcal L_0u_1,v_1\rangle
		+\langle\mathcal L_0u_2,v_2\rangle
		+\displaystyle\int_\Omega\omega(u_2v_1-u_1v_2)\,dx
	}{
		(u_1,v_1)_{L^2}+(u_2,v_2)_{L^2}
	}.
	\]
	
	\noindent\textit{Right level.}
	For $\boldsymbol\phi^*=(0,\phi_1)^T$ and $\mathbf v\in\mathcal S^o$,
	\[
	\mathcal R(\boldsymbol\phi^*,\mathbf v)
	=
	\lambda_1
	+\frac{\displaystyle\int_\Omega\omega\phi_1v_1\,dx}
	{\displaystyle\int_\Omega\phi_1v_2\,dx}
	\ge\lambda_1.
	\]
	Taking $v_1=\varepsilon\psi_1$ and $v_2=\psi_1$ and letting
	$\varepsilon\downarrow0$ gives
	$\inf_{\mathbf v\in\mathcal S^o}\mathcal R(\boldsymbol\phi^*,\mathbf v)=\lambda_1$.
	Conversely, for any $\mathbf u\in\mathcal S\setminus\{\mathbf0\}$, set
	$\mathbf v_\varepsilon:=(\varepsilon\psi_1,\psi_1)^T\in\mathcal S^o$. Then
	\[
	\mathcal R(\mathbf u,\mathbf v_\varepsilon)
	=
	\lambda_1
	+\frac{\displaystyle\int_\Omega\omega(\varepsilon u_2-u_1)\psi_1\,dx}
	{\displaystyle\varepsilon(u_1,\psi_1)_{L^2}+(u_2,\psi_1)_{L^2}}.
	\]
	If $u_2\not\equiv0$, then
	$\limsup_{\varepsilon\downarrow0}\mathcal R(\mathbf u,\mathbf v_\varepsilon)\le\lambda_1$.
	If $u_2\equiv0$, the same inequality is immediate from the last formula,
	possibly with limit $-\infty$. Hence
	$\inf_{\mathbf v\in\mathcal S^o}\mathcal R(\mathbf u,\mathbf v)\le\lambda_1$
	for every $\mathbf u\in\mathcal S\setminus\{\mathbf0\}$, and therefore
	$\overline\lambda=\lambda_1$.
	
	\smallskip
	\noindent\textit{Left level.}
	For $\boldsymbol\psi^*=(0,\psi_1)^T$ and $\mathbf u\in\mathcal S^o$,
	\[
	\mathcal R(\mathbf u,\boldsymbol\psi^*)
	=
	\lambda_1
	-\frac{\displaystyle\int_\Omega\omega u_1\psi_1\,dx}
	{\displaystyle\int_\Omega u_2\psi_1\,dx}
	\le\lambda_1.
	\]
	Taking $u_1=\varepsilon\phi_1$ and $u_2=\phi_1$ and letting
	$\varepsilon\downarrow0$ gives
	$\sup_{\mathbf u\in\mathcal S^o}\mathcal R(\mathbf u,\boldsymbol\psi^*)=\lambda_1$.
	Conversely, for any $\mathbf v\in\mathcal S\setminus\{\mathbf0\}$, set
	$\mathbf u_\varepsilon:=(\varepsilon\phi_1,\phi_1)^T\in\mathcal S^o$. Then
	\[
	\mathcal R(\mathbf u_\varepsilon,\mathbf v)
	=
	\lambda_1
	+\frac{\displaystyle\int_\Omega\omega\phi_1(v_1-\varepsilon v_2)\,dx}
	{\displaystyle\varepsilon(\phi_1,v_1)_{L^2}+(\phi_1,v_2)_{L^2}}.
	\]
	If $v_2\not\equiv0$, then
	$\liminf_{\varepsilon\downarrow0}\mathcal R(\mathbf u_\varepsilon,\mathbf v)\ge\lambda_1$.
	If $v_2\equiv0$, the same inequality is immediate from the last formula,
	possibly with limit $+\infty$. Hence
	$\sup_{\mathbf u\in\mathcal S^o}\mathcal R(\mathbf u,\mathbf v)\ge\lambda_1$
	for every $\mathbf v\in\mathcal S\setminus\{\mathbf0\}$, and therefore
	$\underline\lambda=\lambda_1$.
	
	Finally,
	$(\boldsymbol\phi^*,\boldsymbol\psi^*)_H=\int_\Omega\phi_1\psi_1\,dx>0$.
\end{proof}

\begin{remark}[Relation with the left spectral edge]
	If $\omega(x)\equiv\omega_0$, complexification gives
	\[
	\sigma(\mathcal L)
	=
	\bigl(\sigma(\mathcal L_0)+i\omega_0\bigr)
	\cup\bigl(\sigma(\mathcal L_0)-i\omega_0\bigr).
	\]
	Consequently, whenever $\lambda_1$ is the left spectral edge of
	$\mathcal L_0$, it is also the left spectral edge of $\mathcal L$.
\end{remark}

\medskip

\section{Concluding remarks}

The cone minimax framework developed in this paper is related to several
classical directions in spectral theory, but it is not a direct
reformulation of them. For positive matrices and compact strongly positive
operators, the Perron--Frobenius and Krein--Rutman theories provide a
principal eigenvalue with positive right and left eigenvectors and lead to
minimax-type formulas. In that setting, the present construction is
consistent with the classical picture. Its purpose, however, is different:
the quotient
\[
\mathcal R(u,v)
=
\frac{\langle\mathcal L u,v\rangle}{(\mathcal G u,v)_H}
\]
is treated as a two-variable variational object associated with the pencil
\(\mathcal L-\lambda\mathcal G\), without first passing to a positive
resolvent formulation.

One of the main advantages of this formulation is its a posteriori content.
When the right and left minimax levels coincide, the selected value
\(\hat\lambda\) is not only characterized by a minimax identity, but is also
bracketed by computable one-sided quantities:
\[
\lambda_R(u)
:=
\inf_{v\in\mathcal S^o}\mathcal R(u,v),
\qquad
\lambda_L(v)
:=
\sup_{u\in\mathcal S^o}\mathcal R(u,v),
\qquad
\lambda_R(u)\le \hat\lambda\le \lambda_L(v).
\]
Thus any admissible trial functions \(u\) and \(v\) provide lower and upper
certificates for the minimax spectral level. This is a practical feature of
the cone minimax formula: it gives verifiable bounds without requiring the
complete solution of the eigenvalue problem.

The same point is important for finite-dimensional approximation. The
discrete cone minimax levels are not merely auxiliary Galerkin objects; they
come with their own computable a posteriori bounds and, under the compactness
and positivity assumptions used above, converge to the continuous minimax
level. Moreover, the corresponding normalized discrete right-left eigenpairs
have cluster points which solve the limiting pencil. In this sense, the
method combines a variational characterization with a finite-dimensional
certification procedure.

The perturbation estimates obtained above have a related one-sided character.
Classical analytic perturbation theory gives derivative formulas for isolated
simple eigenvalues. Here the selected value is a minimax characteristic of
the pencil, and its variation is controlled by one-sided quantities rather
than, in general, by a single derivative formula. Thus the theory provides a
robust perturbation band, which is sharp already in simple model examples.

Overall, cone minimax levels provide a variational and computational tool for
selected real spectral values of non-selfadjoint operator pencils, especially
when a positive-resolvent realization is unavailable or unnatural.

\end{document}